\documentclass[a4paper, 10pt, oneside, onecolumn]{article}

\RequirePackage{geometry}
\usepackage[utf8]{inputenc}
\usepackage[T1]{fontenc}
\usepackage{lmodern}
\usepackage{amsmath}
\usepackage{amssymb}
\usepackage{amsfonts}
\usepackage{amsthm}
\usepackage{xcolor}
\usepackage{hyperref}
\usepackage{cleveref} 
\usepackage{mathtools}
\usepackage[autostyle=true]{csquotes}

\hypersetup{
	colorlinks=false,
	pdfborder={0 0 0},
	pdftitle={Global solutions to a haptotaxis system with a potentially degenerate diffusion tensor in two and three dimensions},
	pdfauthor={Frederic Heihoff},
	pdfkeywords={},
	bookmarksopen=true,
}

\renewcommand{\phi}{\varphi}

\newtheorem{base}{Base}[section]
\numberwithin{equation}{section}

\theoremstyle{plain}
\newtheorem{theorem}[base]{Theorem}
\newtheorem{lemma}[base]{Lemma}

\newtheorem{corollary}[base]{Corollary}

\newtheorem{definition}[base]{Definition}

\theoremstyle{definition}

\newtheorem{remark}[base]{Remark}

\newcommand{\R}{\mathbb{R}}
\newcommand{\N}{\mathbb{N}}
\renewcommand{\d}{\,\mathrm{d}}

\newcommand{\grad}{\nabla}
\renewcommand{\div}{\nabla \cdot}
\renewcommand{\L}[1]{{L^{#1}(\Omega)}}
\newcommand{\defs}{\coloneqq}
\newcommand{\sfed}{\eqqcolon}
\newcommand{\stext}[1]{\;\;\text{ #1 }\;\;}
\newcommand{\eps}{\varepsilon}
\newcommand{\loc}{\mathrm{loc}}
\newcommand{\tmax}{{T_{\mathrm{max}}}}

\newcommand{\supp}{\text{supp }}
\newcommand{\D}{\mathbb{D}}
\newcommand{\De}{\mathbb{D}_\eps}

\newcommand{\ue}{u_\eps}
\newcommand{\we}{w_\eps}

\renewcommand{\a}{a}
\newcommand{\logc}{\mu}
\newcommand{\loge}{r}
\newcommand{\divrege}{\beta}
\newcommand{\divregc}{A}

\newcommand{\prote}{s}
\newcommand{\WD}[1]{{W_\D^{1,#1}(\Omega)}}

\newcommand{\LD}[1]{{L_\D^{#1}(\Omega)}}

\newif\ifclarification
\clarificationtrue

\newcommand\numberthis{\addtocounter{equation}{1}\tag{\theequation}}

\makeatletter
\g@addto@macro\bfseries{\boldmath}
\makeatother

\title{Global solutions to a haptotaxis system with a potentially degenerate diffusion tensor in two and three dimensions}
\author{
	Frederic Heihoff\footnote{fheihoff@math.uni-paderborn.de}\\
	{\small Institut f\"ur Mathematik, Universit\"at Paderborn,}\\
	{\small 33098 Paderborn, Germany}
}
\date{}

\begin{document}
	
	\maketitle
	
	\begin{abstract}
		\noindent We consider the potentially degenerate haptotaxis system
		\begin{equation*}
		\left\{
		\begin{aligned}
		u_t &= \nabla \cdot (\mathbb{D} \nabla u + u \nabla \cdot \mathbb{D}) - \chi \nabla \cdot (u\mathbb{D}\nabla w) + \mu u(1-u^{r- 1}), \\
		w_t &= - uw
		\end{aligned}
		\right.
		\end{equation*}
		in a smooth bounded domain $\Omega \subseteq \mathbb{R}^n$, $n \in \{2,3\}$, with a no-flux boundary condition, positive initial data $u_0$, $w_0$ and parameters $\chi > 0$, $\mu > 0$, $r \geq 2$ and $\mathbb{D}: \overline{\Omega} \rightarrow \mathbb{R}^{n\times n}$, $\mathbb{D}$ positive semidefinite on $\overline{\Omega}$. 
		\\[0.5em]
		Our main result regarding the above system is the construction of weak solutions under fairly mild assumptions on $\mathbb{D}$ as well as the initial data, encompassing scenarios of degenerate diffusion in the first equation. As a step in this construction as well as a result of potential independent interest, we further construct classical solutions for the same system under a global positivity assumption for $\mathbb{D}$, which ensures the full regularizing influence of its associated diffusion operator. In both constructions, we naturally rely on the regularizing properties of a sufficiently strong logistic source term in the first equation.
		\\[0.5em]
	\end{abstract}
	
	\section{Introduction}
	As the movement of cells plays a significant role in many biological systems and processes, analyzing the underlying mechanisms can prove useful in understanding these systems and processes themselves. One such process, which is naturally of extensive interest, is the invasive movement of tumor cells into healthy tissue along gradients of tissue density during the progression of certain types of cancer, which is governed by a mechanism generally called haptotaxis (cf.\ \cite{CarterHaptotaxisMechanismCell1967}). Similarly to the efforts made to understand the related process of chemotaxis (cf.\ \cite{BellomoMathematicalTheoryKeller2015}), which models movement along gradients of a diffusive chemical as opposed to non-diffusive tissue, mathematical modeling of haptotaxis has proven to be a fruitful area of study. In both cases, by far the most attention at this point has been paid to approaches employing a Fickian diffusive movement model for the organisms in question, which assumes some homogeneity of the underlying medium. But bolstered by experiments regarding cell aggregation near interfaces between grey and white matter in mouse brains (cf.\ \cite{Burden-GulleyNovelCryoimagingGlioma2011a}), it has recently been suggested that especially in more heterogeneous environments, such as brain tissue, cell movement might be better described by non-Fickian diffusion (cf.\ \cite{Belmonte-BeitiaModellingBiologicalInvasions2013}), which is far less mathematically studied in these taxis settings.
	\\[0.5em]
	In an effort to add to the base of knowledge in this area, we will focus our efforts here on a haptotaxis model of cancer invasion featuring such non-Fickian \emph{myopic diffusion}, which was introduced in \cite{EngwerEffectiveEquationsAnisotropic2016}. More specifically, we consider the system
	\begin{equation}\label{proto_problem}
	\left\{
	\begin{aligned}
	u_t &= \nabla \cdot (\mathbb{D} \nabla u + u \nabla \cdot \mathbb{D}) - \chi \nabla \cdot (u\mathbb{D}\nabla w) + \logc u(1-u^{r- 1}), \\
	w_t &= - uw
	\end{aligned}
	\right.
	\end{equation}
	in a smooth bounded domain $\Omega \subseteq \R^n$, $n \in \{2,3\}$, with a no-flux boundary condition and appropriate parameters $\chi > 0$, $\mu > 0$, $r \geq 2$ and $\mathbb{D}: \overline{\Omega} \rightarrow \mathbb{R}^{n\times n}$, $\D$ positive semidefinite on $\overline{\Omega}$. The first equation models the invading cancer cells moving according to the aforementioned myopic diffusion, which is represented by the term $\div(\mathbb{D} \nabla u + u \nabla \cdot \mathbb{D})$, as well as according to haptotaxis, which is represented by the term $- \chi \nabla \cdot (u\mathbb{D}\nabla w)$. Apart from this, the equation further incorporates a logistic source. The second equation models the remaining healthy tissue cells and only features a consumption term. 
	\\[0.5em]
	The key feature of interest in the above system from both an application as well as a mathematical perspective is of course the parameter matrix $\D$, which represents a space dependent coupled diffusion and taxis tensor. In practice, this tensor can be derived from the underlying tissue structure by employing direct imaging methods (cf.\ \cite{EngwerEffectiveEquationsAnisotropic2016}) and represents the influence of said underlying structure on the movement of cells through it. To account for situations of both locally very dense as well as locally very sparse tissue, which both occur in concrete applications and hinder cell movement significantly, we allow $\D$ to be potentially degenerate. Notably in one dimension, solutions to a closely related system with degenerate diffusion have already been shown to reflect the aggregation behavior in interface regions seen in experiments (cf.\ \cite{Burden-GulleyNovelCryoimagingGlioma2011a}, \cite{WinklerSingularStructureFormation2018}) while to our knowledge in systems with non-degenerate diffusion long-time behavior results seem to generally be restricted to homogenization (cf.\ e.g.\ \cite{WangLargeTimeBehavior2016}). This seems to indicate that models of this kind featuring degenerate diffusion could potentially be a better representation of real world behavior. As such, the development of methods to cope with the challenges of this model, namely the reduced regularizing effect of the degenerate diffusion operator and the destabilizing effects of taxis, while still allowing for a sufficiently large class of matrices to enable the modeling of real world scenarios seems to be a worthwhile endeavor, which to our knowledge has thus far only been addressed in one dimension. Thus, the aim of this paper is the investigation of the apparently still open question whether solutions exist in two and three dimensions even if the diffusion operator is degenerate.

	\paragraph{Results.}Our main results regarding the haptotaxis model described above are twofold: First, we establish the existence of global classical solutions given a uniform positivity condition for $\D$, which allows us to basically treat it as we would any other elliptic diffusion operator, as well as a condition ensuring sufficient regularizing influence of the logistic source term. Second, we establish that it is still possible to construct fairly standard weak solutions under much more relaxed conditions for $\D$. More specifically, we drop the assumption that $\D$ must be globally positive in $\overline{\Omega}$ and replace it with a set of assumptions much more tailored to our methods for constructing said weak solutions, which are strictly weaker than the prior positivity assumption in allowing for matrices that are in some (small) parts of $\Omega$ only positive semidefinite. 
	\\[0.5em]
	Given that the definitions necessary to properly formulate the above results take up significant space, we will not go into more detail here but instead refer the reader to the very next section for the pertinent details regarding said results. As an addition to stating precise versions of our results in the next section, we will further discuss the prototypical examples of a matrix with a single point degeneracy as well as a matrix with degeneracies on a manifold of higher dimension and derive some conditions, under which they still allow for the construction of weak solutions. We do this to help build some additional intuition in parallel to the rather abstract regularity properties introduced in said section as well as to illustrate that our results can work for some scenarios with real world relevance such as e.g.\ a domain divided by an impenetrable membrane.
	\paragraph{Approach.}Let us now give a brief sketch of the methods employed to achieve our two main results. 
	\\[0.5em]
	For our classical existence result, we begin by using standard contraction mapping methods to gain local solutions with an associated blow-up criterion as the operator in the first equation is strictly elliptic for globally positive matrices $\D$. We then immediately transition to analyzing the function $a \defs u e^{-\chi w}$, which together with $w$ solves the closely related problem  (\ref{problem_a}). We do this because, in a sense, this transformation eliminates the problematic cross-diffusive term from the first equation by integrating it into the function $a$ and its diffusion operator. Using a fairly classical Moser-type iteration argument, we then establish an $L^\infty(\Omega)$ bound for $a$, which translates back to $u$. Using this bound combined with two testing procedures then yields a further $W^{1,4}(\Omega)$ bound for $w$, which together with the already established bound for $u$ is sufficient to ensure that finite time blow-up is in fact impossible in two and three dimensions and thus completes the proof of our first result.
	\\[0.5em]
	Regarding our second result, we begin by approximating the initial data, the matrix $\D$ as well as the logistic source term in such a way as to make the already established global classical existence result applicable to the in this way approximated versions of (\ref{proto_problem}). For the family of solutions $(\ue, \we)$, $\eps \in (0,1)$, gained in this fashion, we then establish a bound of the form
	\[
		\int_\Omega \ue\ln(\ue) + \int_\Omega \frac{\grad \we\cdot \De \grad \we}{\we} + \int_0^t \int_\Omega \frac{\grad \ue \cdot \De \grad \ue}{\ue} + \int_0^t\int_\Omega \ue^{r+\eps} \ln(\ue) \leq C
	\]
	by way of an energy-type inequality, which already proved useful in the one-dimensional case discussed in \cite{WinklerSingularStructureFormation2018}. Using this as a baseline, we then derive the bounds necessary for applications of the Aubin--Lions compact embedding lemma to gain our desired weak solutions as limits of the approximate ones.
	\paragraph{Prior work.} As haptotaxis models (cf.\ \cite{WangReviewQualitativeBehavior2020} for a general survey) as well as the closely related chemotaxis models (cf.\ \cite{BellomoMathematicalTheoryKeller2015} for a general survey) have been extensively studied in many possible variations since the introduction of their progenitor in the seminal 1970 paper by Keller and Segel (cf.\ \cite{keller1970initiation}), there is of course a lot of prior art available regarding global existence theory for said models. While it is certainly out of scope for this paper to cover prior results in their entirety, we will nonetheless give an overview of some notable ones.
	\\[0.5em] 
	Let us first note that for the one-dimensional case, where $\D$ simplifies to a real-valued function, there are already some results available for a variant of our scenario without a logistic source term (including potential spacial degeneracy) dealing with existence theory as well as long time behavior (cf.\ \cite{WinklerSingularStructureFormation2018},  \cite{WinklerRefinedRegularityStabilization2020} and \cite{WinklerGlobalWeakSolutions2017}). Weak solutions have also been constructed in very similar haptotaxis systems featuring porous-medium type and signal-dependent degeneracies as opposed to spacial ones (cf.\ \cite{ZhigunGlobalExistenceDegenerate2016}).
	\\[0.5em]  
	Regarding haptotaxis system with non-degenerate diffusion operators, e.g.\ $\D \equiv 1$ in our system, global existence and sometimes boundedness theory has been studied in various closely related settings (cf.\ \cite{CaoBoundednessThreedimensionalChemotaxishaptotaxis2016},  \cite{LiBoundednessChemotaxishaptotaxisModel2016}, \cite{LitcanuAsymptoticBehaviorGlobal2010}, 
	\cite{TaoBoundednessStabilizationMultidimensional2014},
	\cite{WalkerGlobalExistenceClassical2006},  \cite{WangBoundednessHigherdimensionalChemotaxishaptotaxis2016}, \cite{XiangNewResult2D2019}). Notably, these systems often feature an additional equation modeling a diffusive (potentially attractive) chemical and the fixed parameter choice $r = 2$ for the logistic term in addition to the more regular diffusion. In many of these scenarios, it has further been established that solutions converge to their constant steady states (cf.\ \cite{LiBoundednessAsymptoticBehavior2015}, \cite{LitcanuAsymptoticBehaviorGlobal2010},
	\cite{PangAsymptoticBehaviorSolutions2019}, \cite{TaoBoundednessStabilizationMultidimensional2014}, \cite{WangLargeTimeBehavior2016},  \cite{ZhengLargeTimeBehavior2019}) under varied but sometimes restrictive assumptions. There has also been some analysis of haptotaxis with tissue remodeling, which is represented in the model by some additional source terms in the equation for $w$ (cf.\  \cite{PangGlobalExistenceTwodimensional2017}, \cite{TaoGlobalExistenceHaptotaxis2011}, \cite{TaoEnergytypeEstimatesGlobal2014}).
	\\[0.5em]
	Apart from haptotaxis models, there has also been significant analysis of chemotaxis models featuring degenerate diffusion (cf.\ \cite{EberlAnalysisDegenerateBiofilm2014}, \cite{LaurencotChemotaxisModelThreshold2005}, \cite{XuChemotaxisModelDegenerate2020} including degeneracies depending on the cell density itself).
	\\[0.5em]
	Lastly, let us just briefly mention that the regularizing effects of logistic source terms we rely on in this paper have already been very well-documented in various chemotaxis systems (cf.\ \cite{LankeitEventualSmoothnessAsymptotics2015}, \cite{WinklerBoundednessHigherdimensionalParabolicparabolic2010} among many others) as well as haptotaxis systems (cf.\ \cite{TaoGlobalClassicalSolutions2020}). 
	\section{Main Results and Related Definitions}
	\label{section:main_results}
	As already alluded to in the introduction, we will focus our attention in this paper on the system
	\begin{equation}\label{problem}
	\left\{
	\begin{aligned}
	u_t &= \div (\D \grad u + u \div \D) - \chi \div (u\D \grad w) + \logc u(1-u^{\loge - 1}) \;\;\;\; &&\text{ on } \Omega\times(0,\infty), \\
	w_t &= - uw \;\;\;\; &&\text{ on } \Omega\times(0,\infty),  \\
	(\D \grad u) \cdot \nu &= \chi (u\D \grad w) \cdot \nu - u (\div \D) \cdot \nu \;\;\;\; &&\text{ on } \partial\Omega\times(0,\infty),\\
	u(\cdot, 0) &= u_0, \;\; w(\cdot, 0) = w_0 \;\;\;\; &&\text{ on } \Omega
	\end{aligned}
	\right.
	\end{equation}
	in a smooth bounded domain $\Omega \subseteq \R^n$, $n \in \{2,3\}$, with parameters $\chi > 0$, $\mu > 0$, $r \geq 2$, $\D: \overline{\Omega} \rightarrow \R^{n\times n}$, $\D$ positive semidefinite on $\overline{\Omega}$, and some initial data $u_0, w_0 : \Omega \rightarrow [0,\infty)$.
	\\[0.5em]
	Our results concerning this system are twofold. We will first derive the following existence result concerning global classical solutions in two and three dimensions under the assumptions that $\D$ and the initial data are sufficiently regular, $\D$ is positive definite on $\overline{\Omega}$ and the logistic source term is sufficiently strong.
	\begin{theorem} \label{theorem:classical_solution}
		Let $\Omega \subseteq \R^n$, $n\in\{2,3\}$, be a bounded domain with a smooth boundary, $\chi \in (0,\infty)$, $\mu \in (0,\infty)$, $r\in[2,\infty)$ and $\D\in C^2(\overline{\Omega}; \R^{n\times n})$. We further assume that $\D$ is positive definite on $\overline{\Omega}$ and satisfies $(\div \D) \cdot \nu = 0$ on $\partial \Omega$. Let $u_0, w_0 \in C^{2+\vartheta}(\overline{\Omega})$, $\vartheta \in (0,1)$, be some initial data with $u_0, w_0 > 0$ on $\overline{\Omega}$ and $(\D \grad u_0)\cdot\nu = (\D \grad w_0)\cdot \nu= 0$ on $\partial \Omega$.
		\\[0.5em]
		If either $r > 2$ or $\logc \geq \chi \|w_0\|_\L{\infty}$, then there exist positive functions $u, w \in C^{2,1}(\overline{\Omega}\times[0,\infty))$ such that $(u,w)$ is a global classical solution to (\ref{problem}) with initial data $(u_0, w_0)$.
	\end{theorem}\noindent
	This result, while of course also of independent interest, will then serve as a building block for the construction of weak solutions to the same system under much more relaxed restrictions on $\D$ and the initial data. Chiefly, global positivity of the matrix $\D$ is not necessarily needed anymore and is instead replaced by a set of much weaker but more specific regularity assumptions.
	\\[0.5em]
	The first such regularity property concerns the divergence of $\D$ (applied column-wise) and how it can be estimated by the (potentially degenerate) scalar product induced by $\D$.
	\begin{definition}\label{definition:div_regularity} Let $\Omega \subseteq \R^n$, $n\in\N$, be a bounded domain with a smooth boundary.
		We then say a positive semidefinite $\D = (\D_1 \,\dots\, \D_n) \in L^1(\Omega; \R^{n \times n})$ with $\div \D \defs (\div \D_1 \,, \dots, \div \D_n)\in L^1(\Omega; \R^n)$ allows for a \emph{divergence estimate} with exponent $\divrege \in [\frac{1}{2},1)$ if there exists $\divregc \geq 0$ such that
		\begin{equation} \label{regularity_div}
		\int_\Omega \left|(\div \D) \cdot \Phi\right| \leq \divregc \left( \int_\Omega  \left| \Phi \cdot \D \Phi \right|^\divrege + 1 \right)
		\end{equation}
		for all $\Phi \in C^0(\overline{\Omega};\R^n)$.
	\end{definition}
	\begin{remark}\label{remark:div_regularity_consequence}
		Note that if $\D \in C^0(\overline{\Omega};\R^{n\times n})$, $\D$ allowing for a divergence estimate with exponent $\beta \in (\frac{1}{2}, 1)$ implies that $\div \D \in L^\frac{2\beta}{2\beta - 1}(\Omega; \R^n) \subseteq L^2(\Omega; \R^n) $. This stems from the fact that the estimate (\ref{regularity_div}) essentially means that the functional $\Phi \mapsto \int_\Omega (\div \D) \cdot \Phi$ is an element of $(L^{2\beta}(\Omega; \R^n))^*$, which is isomorphic to $L^{\frac{2\beta}{2\beta - 1}}(\Omega;\R^n)$.
	\end{remark}
	\begin{remark}\label{remark:div}
		It is fairly easy to verify that any smooth, positive definite $\D$ allows for such an estimate with the optimal exponent $\beta = \frac{1}{2}$.  
		Let us therefore now briefly illustrate that the above property is also achievable for less regular $\D$, which are e.g.\ at some points in $\Omega$ only positive semidefinite, by giving some examples. While we will not necessarily fully explore these examples and leave out some of the more cumbersome corner cases for ease of presentation, they will accompany us throughout this section as a tool to give some intuition for later introduced definitions as well as to give concrete examples for degenerate cases in which weak solutions can still be constructed.
		\\[0.5em]
		We will first take a look at the prototypical case of a matrix-valued function $\D_1$ on a ball with a single degenerate point in the origin, or more precisely we will consider $\D_1(x) \defs |x|^\prote I$ on $\Omega \defs B_1(0) \subseteq \R^n$, $n \in \N$, $I$ being the identity matrix and $\prote$ being some positive real number. 
		\\[0.5em] 
		As $\div \D_1(x) = \grad (|x|^\prote) = \prote |x|^{\prote - 2}x$ almost everywhere, we can estimate 
		\begin{align*}
		\int_\Omega \left|(\div \D_1) \cdot \Phi\right| &\leq \prote\int_\Omega |x|^{\prote - 1}\left|\Phi\right| = \prote \left\| |x|^{\frac{\prote}{2} - 1} (\Phi \cdot |x|^{\prote} I \Phi)^\frac{1}{2} \right\|_\L{1} \\
		&\leq s \left\| |x|^{\frac{\prote}{2} - 1} \right\|_\L{\frac{2\beta}{2\beta - 1}} \left\|(\Phi \cdot |x|^{\prote} I \Phi)^\frac{1}{2} \right\|_\L{2\beta}\\
		&\leq s  \left\| |x|^{\frac{s}{2} - 1} \right\|_\L{\frac{2\beta}{2\beta - 1}} \left( \int_\Omega  \left( \Phi \cdot \D_1 \Phi \right)^\beta  + 1 \right)
		\end{align*}
		for all $\beta \in (\frac{1}{2}, 1)$ and $\Phi \in C^0(\overline{\Omega};\R^n)$ using the Hölder inequality as well as Young's inequality. As $|x|^{\frac{\prote}{2} - 1} \in \L{\frac{2\beta}{2\beta - 1}}$ if and only if $\frac{2\beta}{2\beta - 1} (\frac{s}{2} - 1) > - n$, the prototypical case discussed above fulfills the divergence estimate for all $\beta \in (\frac{n}{s  - 2 + 2n},1) \cap (\frac{1}{2}, 1)$. Note that for $s > 2 - n$, which in two or more dimensions is always ensured, the set $(\frac{n}{s  - 2 + 2n},1) \cap (\frac{1}{2}, 1)$ is never empty and therefore our prototypical example always has the discussed property for all positive $s$ and some appropriate $\beta$. 
		\\[0.5em]
		To illustrate that our framework also supports analysis of singularities occurring on higher dimensional manifolds, let us further consider the similar prototypical example $\D_2(x_1, \dots, x_n) \defs |x_1|^s I$ on the same set $\Omega$ with $s$ now being a real number greater than 1. As here $\div \D_2(x_1, \dots, x_n) = (s |x_1|^{s-2}x_1, 0, \dots, 0)$ almost everywhere, we gain that $\D_2$ has the property laid out in \Cref{definition:div_regularity} for all $\beta \in (\frac{1}{s}, 1) \cap (\frac{1}{2}, 1)$ by a similar argument as for the previous example.
		\\[0.5em] 
		As to be expected in both cases, smaller values of $s$ result in the divergence estimate only holding for ever larger exponents $\beta$. As we will see in our theorem regarding the existence of weak solutions at the end of this section, these larger values of $\beta$ will necessitate stronger regularizing influence from the logistic source term to compensate.
	\end{remark}
	\noindent
	Before we can now approach the second regularity property of this section as well as properly defining what we in fact mean by weak solutions in this paper, we need to first introduce a set of function spaces. Said spaces are generally fairly straightforward generalizations of standard Sobolev and Lebesgue spaces incorporating $\D$ as well as some spaces derived from them, which are more specific to our setting. For a more thorough discussion of e.g.\ the degenerate Sobolev spaces introduced below, we refer the reader to \cite{SawyerDegenerateSobolevSpaces2010}.
	\\[0.5em]
	We will further take the introduction of said spaces as an opportunity to present some of their most important properties for our purposes immediately after defining them. 
	\begin{definition}\label{definition:spaces}
		Let $\Omega\subseteq\R^n$, $n\in\N$, be a bounded domain with a smooth boundary and $p\in [1,\infty)$. \\[0.5em] 
		We then define the Sobolev-type space
		\[
			W^{1,p}_\mathrm{div}(\Omega;\R^{n\times n}) \defs \left\{ M \in L^p(\Omega; \R^{n\times n}) \; \middle| \; \div M \in L^p(\Omega; \R^n) \right\}
		\]
		with the norm
		\[
			\|M\|_{W^{1,p}_\mathrm{div}(\Omega;\R^{n\times n})} \defs \|M\|_{L^p(\Omega; \R^{n\times n})} + \|\div M\|_{L^p(\Omega; \R^n)}.
		\]
		Herein, the divergence of a square matrix $M = (M_1\, \dots \, M_n)$ is defined as $\div M \defs (\div M_1,\, \dots, \div M_n)$.
		\\[0.5em]
		Let now $\D \in C^0(\overline{\Omega};\R^{n\times n})$ be positive semidefinite everywhere. We then define the Lebesgue-type space $\LD{p}$ as the set of all measurable $\R^n$-valued functions $\Phi$ on $\Omega$ 
		with finite seminorm
		\[
		\| \Phi \|_{\LD{p}} \defs \left(\int_\Omega \left( \Phi \cdot \D \Phi \right)^\frac{p}{2} \right)^\frac{1}{p}
		\]
		modulo all of those functions with $\| \Phi \|_{\LD{p}} = 0$ in the same vain as the standard Lebesgue spaces.
		\\[0.5em]
		Furthermore, we define the Sobolev-type spaces $\WD{p}$ as the completion of $C^\infty(\overline{\Omega})$ in the norm 
		\[
		\| \phi \|_\WD{p} \defs \| \phi \|_\L{p} + \| \grad \phi \|_{\LD{p}} 
		\]
		in the same vain as the standard Sobolev spaces. It is straightforward to see that each space $\WD{p}$ can be interpreted as a subspace of $\L{p} \times \LD{p}$ in a natural way and thus elements of these spaces can be written as tuples $(\phi, \Phi)$. As such, there exist the natural continuous projections 
		\[P_1 : \WD{p} \rightarrow L^p(\Omega) \stext{ and } P_2: \WD{p} \rightarrow \LD{p}\] 
		associated with this representation.
	\end{definition}
	\begin{remark}
		For a more comprehensive exploration of these spaces and their properties see e.g.\ \cite{SawyerDegenerateSobolevSpaces2010}. 
		\\[0.5em]
		We will now give a brief overview of the properties the above spaces retain from the standard Sobolev and Lebesgue spaces as well as some of the differences. As most of the proofs translate directly from standard Sobolev theory or are laid out in \cite{SawyerDegenerateSobolevSpaces2010}, we will only list the properties we are interested in without extensive argument.
		\\[0.5em]
		First of all by construction, $W_\mathrm{div}^{1,p}(\Omega;\R^{n \times n})$, $\LD{p}$ and $\WD{p}$ are Banach spaces, which are reflexive if $p \in (1,\infty)$, by essentially the same arguments as for the standard Sobolev and Lebesgue spaces and, for $p = 2$, they are in fact Hilbert spaces with the natural inner products. It is further easy to see that, if $(\phi, \Phi)$ is a strong or weak limit of a sequence $(\phi_n, \Phi_n)_{n\in\N} \subseteq \WD{p}$, the function $\phi\in\L{p}$ coincides with the pointwise almost everywhere limit of the sequence $(\phi_n)_{n\in\N}$ if it exists due to $P_1$ being continuous regarding both topologies and well-known results about strong and weak convergence in $\L{p}$.
		\\[0.5em]
		As opposed to the classical Sobolev spaces, the spaces $\WD{p}$ can not necessarily be understood as subspaces of the spaces $\L{p}$ because their equivalents to the weak gradients in the classical Sobolev spaces are not necessarily unique here, meaning essentially that $P_1$ is not always injective.\ (For an example of this, see \cite[p.\ 1877]{SawyerDegenerateSobolevSpaces2010}). Given that this can be problematic when deriving analogues to the (compact) embedding properties of Sobolev spaces for our weaker variants, let us now briefly note that, under sufficient regularity assumptions for $\D$, the spaces $\WD{p}$ do in fact embed into the spaces $\L{p}$ again. In particular if $p = 2$, which is the parameter choice we are most interested here, this is the case if $\sqrt{\D} \in W_{\mathrm{div}}^{1,2}(\Omega; \R^{n \times n})$ according to Lemma 8 from \cite{SawyerDegenerateSobolevSpaces2010}. 	
		\\[0.5em]
		While it presents a slight abuse of notation, we will in a similar fashion to \cite{SawyerDegenerateSobolevSpaces2010} use $\phi$ to mean $P_1(\phi) \in \L{p}$ for elements $\phi \in \WD{p}$ when unambiguous and generally use the convention $\grad \phi = P_2(\phi)$ even if $\grad \phi$ is not necessarily the actual weak derivative. We will further often write
		\[
		\|\phi\|_\WD{p} \stext{ for } \|(\phi,\grad \phi)\|_\WD{p}
		\]
		to simplify the notation in later arguments.
		If $\phi$ is additionally an element of $C^1(\overline{\Omega})$, we will always assume $\grad \phi$ to be equal to the classical derivative, of course.
	\end{remark}
	\noindent Having established these function spaces, we can now clearly state the second and last regularity property for $\D$ we are interested in. It is a simple compact embedding property, which is mainly used in this paper to facilitate application of the well-known Aubin--Lions lemma. 
	\begin{definition}\label{definition:comp_regularity} Let $\Omega \subseteq \R^n$, $n\in\N$, be a bounded domain with a smooth boundary.
		We say a positive semidefinite $\D \in C^0(\overline{\Omega};\R^{n \times n})$ allows for a \emph{compact $L^1(\Omega)$ embedding} if $\WD{2}$ embeds compactly into $\L{1}$.
	\end{definition}
	\begin{remark}
		Let us briefly note that any $\D$, which is equal to zero on any open subset $U$ of $\Omega$, cannot fulfill the property laid out in \Cref{definition:comp_regularity} as it is well documented that $L^2(U)$, which is equal to $W^{1,2}_\D(U)$ in this case, does not embed compactly into $L^1(U)$. 
	\end{remark}
	\noindent We will now give some additional criteria for the above compact embedding property to not only make our results easier to use in application but also to help us prove that both of the examples discussed in \Cref{remark:div} in fact fulfill it.
	\begin{lemma}\label{lemma:comp_regularity_criteria}
		Let $\Omega \subseteq \R^n$, $n\in\N$, be a bounded domain with a smooth boundary and $N \subseteq \Omega$ be a relatively closed set in $\Omega$ with measure zero. Let then \[
			\Omega_{N,\eps}  \defs \left\{ x \in \Omega \; \middle| \; \mathrm{dist}(x, N\cup\partial\Omega) > \eps \right\}
		\] 
		and let 
		$\D \in C^0(\overline{\Omega};\R^{n \times n})$ be positive semidefinite and fulfill $\sqrt{\D} \in W^{1,2}_{\mathrm{div}}(\Omega;\R^{n \times n})$. If 
		\begin{enumerate}
			\item $\WD{2}$ embeds compactly into $L^1_\loc(\Omega\setminus N)$ or
			\item $\D$ is positive definite on $\Omega\setminus N$ and there exists $\eps_0 > 0$ such that $W^{1,2}(\Omega_{N,\eps})$ embeds compactly into $L^1(\Omega_{N,\eps})$ for all $\eps \in (0,\eps_0)$ or
			\item $\D$ is positive definite on $\Omega\setminus N$ and there exists $\eps_0 > 0$ such that $\Omega_{N,\eps}$ has a Lipschitz boundary for all $\eps \in (0,\eps_0)$,
		\end{enumerate}
		then $\D$ allows for a compact $L^1(\Omega)$ embedding.
	\end{lemma}
	\begin{proof}
		Due to our assumption that $\sqrt{\D} \in W^{1,2}_\mathrm{div}(\Omega;\R^{n\times n})$, Lemma 8 from \cite{SawyerDegenerateSobolevSpaces2010} immediately yields that the projection $P_1 :  \WD{2}\rightarrow L^2(\Omega) \subseteq L^1(\Omega)$ from \Cref{definition:spaces} is injective and thus provides us with a continuous embedding of $\WD{2}$ into $L^1(\Omega)$. It thus only remains to show that this embedding is in fact compact given the various criteria outlined above.
		\\[0.5em]
		To do this, we first fix a bounded sequence $(\varphi_k)_{k\in\N} \subseteq \WD{2}$. We then only need to construct a subsequence of $(\varphi_k)_{k\in\N}$ that converges in $L^1(\Omega)$ to some function $\varphi$ to prove our desired outcome. As it is further possible to find another sequence $(\psi_k)_{k\in\N} \subseteq C^\infty(\overline{\Omega})$ such that $\|\varphi_k - \psi_k\|_\L{1} \leq \sqrt{|\Omega|} \|\varphi_k - \psi_k\|_\WD{2} \leq \frac{1}{k}$ by definition of $\WD{2}$, we can further assume that $\varphi_k \in C^\infty(\overline{\Omega})$ for all $k$ without loss of generality.
		\\[0.5em]
		If $\WD{2}$ now embeds compactly into $L^1_\loc(\Omega\setminus N)$, we can choose a subsequence $(\varphi_{k_j})_{j\in\N}$ and function $\varphi: \Omega \rightarrow \R$ such that $\varphi_{k_j} \rightarrow \varphi$ in all $L^1(\Omega_{N,\eps})$, $\eps > 0$, as $j \rightarrow \infty$. As by our assumptions $N\cup\partial\Omega$ is closed and thus $\Omega\setminus N = \bigcup_{k\in\N} \Omega_{N, 1/k}$, we can then employ a standard diagonal sequence argument to gain yet another subsequence, which we will again call $(\varphi_{k_j})_{j\in\N}$ for convenience, with the property that $\varphi_{k_j} \rightarrow \varphi$ almost everywhere in $\Omega\setminus N$  and thus almost everywhere in all of $\Omega$ as $j \rightarrow \infty$ because $N$ is a null set. Given that the thus constructed subsequence is further bounded in $L^2(\Omega)$ due to it being bounded in $\WD{2}$, we can use Vitali's theorem and the de La Valleé Poussin criterion for uniform integrability (cf.\ \cite[pp.\ 23-24]{DellacherieProbabilitiesPotential1978}) to conclude that $\varphi_{k_j} \rightarrow \varphi$ in $L^1(\Omega)$ as well, yielding the first part of our result.
		\\[0.5em]
		If $\D$ is positive definite on $\Omega\setminus N$, then for every $\eps > 0$ there exists $K(\eps) > 0$ such that $\D > K(\eps)$ on $\Omega_{N,\eps}$ due to the continuity of $\D$ and the fact that $\overline{\Omega_{N,\eps}} \subseteq \Omega\setminus N$ is compact. Thus, the norms of the spaces $W^{1,2}(\Omega_{N,\eps})$ and $W_\D^{1,2}(\Omega_{N,\eps})$ are equivalent. As such, the sequence $(\varphi_k)_{k\in\N}$ is bounded in all of the spaces $W^{1,2}(\Omega_{N,\eps})$, $\eps > 0$. Due to our further assumption that there exists $\eps_0 > 0$ such that $W^{1,2}(\Omega_{N,\eps})$ embeds compactly into $L^1(\Omega_{N,\eps})$ for all $\eps \in (0,\eps_0)$ and the fact that any compact set $K\subseteq \Omega\setminus N$ is a subset of some $\Omega_{N,\eps}$ as another consequence of $\Omega\setminus N$ being equal to $\bigcup_{\eps > 0} \Omega_{N, \eps}$, a standard diagonal sequence argument yields a subsequence along which the functions $\varphi_{k}$ converge to some $\varphi$ in $L^1_\loc(\Omega\setminus N)$. Combining this with the arguments from the previous paragraph then yields the second part of our result.
		\\[0.5em]
		To now complete the proof, we first note that \cite[Theorem 6.3]{AdamsSobolevSpaces2003} states that a Lipschitz boundary condition for the sets $\Omega_{N,\eps}$ ensures the Sobolev embedding necessary for our second result and thus the third result follows directly from the second.
	\end{proof}
	\begin{remark}\label{remark:comp}
	Going briefly back to the examples introduced in \Cref{remark:div}, we see that $\div \sqrt{\D_1(x)} = \frac{s}{2}|x|^{\frac{s}{2} - 2}x$, $s > 0$, and $\div \sqrt{\D_2(x_1, \dots, x_n)} = (\frac{s}{2}|x_1|^{\frac{s}{2} - 2}x_1, 0, \dots, 0)$, $s > 1$, are both elements of $L^2(\Omega; \R^n)$ and thus $\sqrt{\D_1}, \sqrt{\D_2} \in W^{1,2}_\mathrm{div}(\Omega; \R^{n\times n})$ in dimensions two or higher. Furthermore due to the fairly straightforward geometry of the degeneracy set $N$ in both cases, it is easy to verify that both examples also fulfill the third criterion in \Cref{lemma:comp_regularity_criteria} and thus both $\D_1$ and $\D_2$ allow for a compact $L^1(\Omega)$ embedding in accordance with \Cref{definition:comp_regularity}.
	\end{remark}
	\noindent
	While we have now invested some effort into formalizing the restrictions on $\D$ necessary for our later construction of weak solutions, we have yet to clarify what we in fact mean by a weak solution to (\ref{problem}). Let us now rectify this in the following definition.
	\begin{definition}\label{definition:weak_solution}
		Let $\Omega \subseteq \R^n$, $n\in\N$, be a bounded domain with a smooth boundary and let $\chi > 0$, $\mu > 0$ and $r\in[1,\infty)$. Let $\D \in W_\mathrm{div}^{1,q}(\Omega; \R^{n\times n}) \cap C^0(\overline{\Omega}; \R^{n \times n})$, $q \in (1,\infty)$ and $p \defs \max(2,r,\frac{q}{q-1})$. Let further $u_0, w_0 \in L^1(\Omega)$ be some initial data.
		\\[0.5em]	
		We then call a tuple of functions
		\begin{align*}
		u &\in L_\loc^1([0,\infty);\WD{1}) \cap L_\loc^p(\overline{\Omega}\times[0,\infty)), \\
		w &\in L_\loc^2([0,\infty);\WD{2})
		\end{align*}
		a weak solution of (\ref{problem}) with initial data $u_0$, $w_0$ and the above parameters if
		\begin{align*}
		\int_0^\infty \int_\Omega u \phi_t - \int_\Omega u_0 \phi(\cdot, 0) &= \int_0^\infty \int_\Omega \grad u \cdot \D \grad \phi + \int_0^\infty \int_\Omega u (\div \D) \cdot \grad \phi \\
		&- \chi \int_0^\infty \int_\Omega u \grad w \cdot \D \grad \phi - \mu \int_0^\infty \int_\Omega u(1-u^{r-1})\phi
		\end{align*}
		and 
		\[
		\int_0^\infty \int_\Omega w \phi_t - \int_\Omega w_0 \phi(\cdot, 0) = \int_0^\infty\int_\Omega uw \phi
		\]
		hold for all $\phi \in C_c^\infty(\overline{\Omega}\times[0,\infty))$.
	\end{definition}
	\noindent
	As we have at this point clearly defined the target and some of the preconditions, let us now outright state the second main theorem we endeavor to prove in this paper.
	\begin{theorem}\label{theorem:weak_solutions}
		Let $\Omega \subseteq \R^n$, $n\in\{2,3\}$, be a bounded domain with a smooth boundary, $\chi \in (0,\infty)$, $\mu \in (0,\infty)$, $\beta \in [\frac{1}{2}, 1)$, $r\in[2,\infty)$ with $\frac{\beta}{1-\beta} \leq r$ and $\D \in W^{1,2}_\mathrm{div}(\Omega;\R^{n\times n})\cap C^0(\overline{\Omega}; \R^{n\times n})$ be positive semidefinite everywhere. Let further $\D$ allow for a divergence estimate with exponent $\divrege$ (cf.\ \Cref{definition:div_regularity}) and let $\D$ allow for a compact $L^1(\Omega)$ embedding (cf.\ \Cref{definition:comp_regularity}). Finally, let $u_0 \in  \L{z[\ln(z)]_+}$ and $w_0 \in C^0(\overline{\Omega})$ be some initial data with  $\sqrt{w_0} \in W^{1,2}(\Omega)$ %Could potentially be $\WD{2}$, if I am not wrong
		and $u_0 \geq 0$, $w_0 \geq 0$ almost everywhere. Here, $\L{z[\ln(z)]_+}$ is the standard Orlicz space associated with the function $z \mapsto z [\ln(z)]_+$.
		\\[0.5em]
		Then there exist a.e.\ non-negative functions 
		\begin{align*}
			u &\in L_\loc^\frac{2r}{r+1}([0,\infty);\WD{\frac{2r}{r+1}}) \cap L_\loc^r(\overline{\Omega}\times[0,\infty)), \numberthis \label{eq:weak_solution_regularity_1} \\
			w &\in L_\loc^2([0,\infty);\WD{2})\cap L^\infty(\Omega\times(0,\infty)) \numberthis \label{eq:weak_solution_regularity_2}
		\end{align*} 
		that are a weak solution to (\ref{problem}) in the sense of \Cref{definition:weak_solution}.
	\end{theorem}
	\begin{remark}
		In light of the above theorem, we take another look at our prototypical examples $\D_1(x) \defs |x|^s I$, $s > 0$, and $\D_2(x_1, \dots, x_n) \defs |x_1|^s I$, $s > 1$, on $\Omega \defs B_1(0)\subseteq \R^n$, $n\in\{2,3\}$. Given the discussion in \Cref{remark:div}, we know that, if we assume
		\begin{equation}\label{eq:remark_div_consequence_1}
			r\geq \frac{\beta}{1-\beta} > \frac{\frac{n}{s-2+2n}}{1 - \frac{n}{s-2+2n}} = \frac{n}{s-2+n},  
		\end{equation}
		then $\D_1$ allows for the necessary divergence estimate with an exponent $\beta$ fulfilling $\frac{\beta}{1-\beta} \leq r$. Similarly if we assume that 
		\begin{equation}\label{eq:remark_div_consequence_2}
			r \geq \frac{1}{s-1},
		\end{equation}
		the same holds true for $\D_2$. Further due to the arguments presented in \Cref{remark:comp}, both $\D_1$ and $\D_2$ allow for a compact $L^1(\Omega)$ embedding. Therefore, the above theorem means that, for sufficiently regular initial data $u_0$, $w_0$ and if either $\D = \D_1$ and $r$ and $s$ satisfy (\ref{eq:remark_div_consequence_1}) or $\D = \D_2$ and $r$ and $s$ satisfy (\ref{eq:remark_div_consequence_2}), weak solutions to (\ref{problem}) in fact exist in two and three dimensions.
	\end{remark}
	
	\section{Existence of Classical Solutions}
	
	As the existence of classical solutions to (\ref{problem}), apart from being an interesting result by its own merits, plays an important role in our construction of their weak counterparts, we will in this section first focus on their derivation. As such, our ultimate goal for this section will be the proof of our first main result, namely \Cref{theorem:classical_solution}. The methods presented here will in many ways mirror those for similar systems with a standard Laplacian as diffusion operator. We mainly verify that the differing elements in our systems do not impede said methods.
	\\[0.5em]
	To this end, we now fix a smooth bounded domain $\Omega\subseteq\R^n$, $n\in\{2,3\}$ and system parameters $\chi \in (0,\infty)$, $\logc \in (0,\infty)$, $\loge \in [2,\infty)$ and $\D \in C^2(\overline{\Omega};\R^{n\times n})$. We further assume that $\D$ is in fact positive definite everywhere and has the property $(\div \D) \cdot \nu = 0$ on $\partial \Omega$. Given these assumptions, we can fix $M \geq 1$ such that
	\begin{equation}\label{eq:classical_D_assumption}
		\frac{1}{M} \leq \D \leq M, \;\;\;\; \|\div \D\|_\L{\infty} \leq M  \stext{ and } \|\div (\div \D)\|_\L{\infty} \leq M.
	\end{equation}
	We also fix some initial data $u_0, w_0 \in C^{2+\vartheta}(\overline{\Omega})$, $\vartheta \in (0,1)$, with $(\D \grad u_0) \cdot \nu = (\D \grad w_0) \cdot \nu = 0$ on $\partial \Omega$ and $u_0 > 0$, $w_0 > 0$ on $\overline{\Omega}$.
	\\[0.5em]
	Comparing the very strong regularity assumptions for $\D$ in this section to the much weaker ones in the following section devoted to the construction of weak solutions, the question why the gap in assumed regularity between these sections is as large as it is naturally presents itself. Let us therefore briefly address this issue. It is certainly possible to derive most of the a priori estimates, which are used in this section to argue that blow-up of local solutions is impossible, under similarly specific regularity assumptions as seen in \Cref{definition:div_regularity} or \Cref{definition:comp_regularity} (albeit with some additions). But generalizing the theory employed by us to first gain said local solutions with less regular $\D$ would necessitate Schauder and semigroup theory for potentially very degenerate operators, which is out of scope for this paper. Furthermore, we think that this result is already of interest in and of itself.
	\subsection{Existence of Local Solutions}
	After this introductory paragraph giving our rational for the assumptions about $\D$ in this section, we will now focus on the construction of local solutions to the system (\ref{problem}) as a first step in constructing global ones. As for a positive definite matrix $\D$, the diffusion operator in the first equation is strictly elliptic and therefore accessible to most of the same existence and regularity theory as the Laplacian, we will not go into detail concerning the construction of local solutions but rather refer the reader to a local existence result for a similar haptotaxis system with our operator replaced by the Laplacian in \cite{TaoGlobalClassicalSolutions2020}. 
	\begin{lemma}\label{lemma:local_solution}
		There exist $\tmax \in (0,\infty]$ and positive functions $
			u,w \in C^{2,1}(\overline{\Omega}\times[0,\tmax))
		$ such that $(u,w)$ is a classical solution to (\ref{problem}) on $\overline{\Omega}\times(0,\tmax)$ with initial data $(u_0, w_0)$ and satisfies the following blow-up criterion:
		\begin{equation} \label{eq:blowup}
		\text{ If } \tmax < \infty,  \text{ then } \limsup_{t\nearrow \tmax} \left( \|u(\cdot, t)\|_\L{\infty} + \|w(\cdot, t)\|_{W^{1,n+1}(\Omega)} \right) = \infty.
		\end{equation}
	\end{lemma}
	\noindent For ease of further discussion, we now fix such a maximal local solution $(u,w)$ on $(0,\tmax)$ with initial data $(u_0, w_0)$ and the parameters as stated in the above introductory paragraphs.
	\\[0.5em]
	Before diving into the derivation of more substantial bounds for the above solution, we derive a straightforward mass bound for the first solution component as well as an $L^\infty(\Omega)$ bound for the second solution component. These bounds will not only prove useful when ruling out blow-up in this section but also serve as a baseline for bounds derived in our later efforts focused on the construction of weak solutions.
	\begin{lemma}\label{lemma:absolute_baseline}
		The inequalities
		\[
		\int_\Omega u(\cdot, t) \leq \mu |\Omega| t + \int_\Omega u_0 \stext{ and } \|w(\cdot, t)\|_\L{\infty} \leq \|w_0\|_\L{\infty}
		\]
		hold for all $t \in (0,\tmax)$.
	\end{lemma}
	\begin{proof}
		Integrating the first equation in (\ref{problem}) and applying partial integration yields
		\[
		\frac{\d}{\d t}\int_\Omega u = \mu \int_\Omega u(1-u^{r-1}) \leq \mu |\Omega|
		\]
		for all $t \in (0,T)$ and therefore immediately give us the first half of our result by time integration. Given that further $w_t \leq 0$ due to the second equation in (\ref{problem}), the second half of our result follows directly as well.
	\end{proof}\noindent
	
	\subsection{A Priori Estimates}
	The next natural step after establishing local solutions with an associated blow-up criterion is of course arguing that finite-time blow-up is impossible and the maximal local solutions were in fact global all along. To do this, we will devote this section to a set of a priori estimates, which increase in strength as the section goes on until they rule out blow-up of both $u$ and $w$. 
	\\[0.5em]
	As is not uncommon in the analysis of these kinds of haptotaxis systems (cf.\ \cite{TaoGlobalClassicalSolutions2020}), we will from now consider the function $a \defs u e^{-\chi w}$ defined on $\overline{\Omega}\times[0,\tmax)$ and its associated initial data $a_0 \defs u_0 e^{-\chi w_0}$ defined on $\overline{\Omega}$ in addition to the actual solutions components $u$ and $w$ themselves. A simple computation then shows that $(a,w)$ is a classical solution of the following related system:
	\begin{equation}\label{problem_a}
	\left\{
	\begin{aligned}
	a_t &=e^{-\chi w} \div (  e^{w\chi}\D\grad a) + e^{-\chi w} \div (a e^{\chi w} (\div \D)) \\
	&\;\;\;\;+ \mu a (1-a^{r-1}e^{\chi(r-1)w}) +\chi a^2 w e^{\chi w} &&\text{ on } \Omega\times(0,\infty),\\
	w_t &= - a e^{\chi w}w  \;\;\;\; &&\text{ on } \Omega\times(0,\infty),  \\
	(\D \grad a) \cdot \nu &= -a (\div \D) \cdot \nu = 0  \;\;\;\; &&\text{ on } \partial\Omega\times(0,\infty),\\
	a(\cdot, 0) &= a_0 > 0, \;\; w(\cdot, 0) = w_0 > 0  \;\;\;\; &&\text{ on } \Omega.
	\end{aligned}
	\right.
	\end{equation}
	The key property of the above system, which makes it so useful for our purposes, is that it in a sense eliminates the taxis term or at least the explicit gradient of $w$ from the first equation (by in a sense integrating it into $a$ and its diffusion operator). This alleviates many of the normal problems associated with the taxis term in testing or semigroup based approaches used to derive a priori estimates. A second useful property of this transformation is that, by definition, bounds that do not involve derivatives are easily translated back from $a$ to $u$ as we will see later. Note however that, as soon as we want to back propagate bounds about the gradient of $a$ to $u$, the complications introduced by the taxis term come back into play, making this transformation much less useful for endeavors of this kind.  
	\\[0.5em]
	We now begin by translating the baseline estimates given in \Cref{lemma:absolute_baseline} to our newly defined function $a$ as we will henceforth focus on $(a,w)$ as our central object of analysis for quite some time. We will further for the foreseeable future work under the assumption that $\tmax < \infty$ as this is exactly the case we want to rule out by leading this assumption to a contradiction with the blow-up criterion.
	\begin{corollary}\label{corollary:absolute_baseline_a}
		If $\tmax < \infty$, there exists $C > 0$ such that
		\[
			\int_\Omega a \leq C
		\]
		for all $t\in(0,\tmax)$.
	\end{corollary}
	\begin{proof}
		As $\int_\Omega a = \int_\Omega u e^{\chi w} \leq e^{\chi \|w\|_\L{\infty}} \int_\Omega u$, this is a direct consequence of \Cref{lemma:absolute_baseline} if $\tmax < \infty$.
	\end{proof}\noindent
	In preparation for a later Moser-type iteration argument for the first solution component $a$ (cf.\ \cite{AlikakosBoundsSolutionsReactiondiffusion1979} and \cite{MoserNewProofGiorgi1960} for some early as well as \cite{FuestBlowupProfilesQuasilinear2020} and \cite{TaoBoundednessQuasilinearParabolicparabolic2012} for some more contemporary examples of this technique), which will later be used to rule out its finite-time blow-up, we will now derive a recursive inequality for terms of the form $\int_\Omega a^p$. This recursion will in fact allow us to estimate each term of the form $\int_\Omega a^p$ by terms of the form $(\int_\Omega a^\frac{p}{2})^2$ with constants independent of $p$, which will prove sufficient to later gain an $L^\infty(\Omega)$ bound for $a$. The method employed to gain said recursion is testing the first equation in (\ref{problem_a}) with $e^{\chi w}a^{p-1}$ followed by some estimates based on the Gagliardo--Nirenberg inequality. 
	\\[0.5em]
	To facilitate this derivation of said recursion, we will from now on assume that the regularizing influence of the logistic source term in the first equation of (\ref{problem}) is sufficiently strong, or more precisely we assume that either $r > 2$ or $\mu$ is sufficiently large in comparison to $\chi$ and the $L^\infty(\Omega)$ norm of $w_0$. However at this point and therefore for the whole of the Moser-type iteration argument, we will not use our assumed restriction to two or three dimensions just yet.
	\begin{lemma} \label{lemma:linfty_recursion} 
	If $\tmax < \infty$ and further $r > 2$ or $\logc \geq \chi \|w_0\|_\L{\infty}$, then there exists a constant $C > 0$ such that
	\[
		\sup_{t\in(0,\tmax)}\int_\Omega a^p \leq C\max\left( \; \int_\Omega a_0^p, \; C^{p+1}, \; p^C \left( \sup_{t\in(0,\tmax)}\int_\Omega a^\frac{p}{2} \right)^2 \; \right)	
	\]
	for all $p \geq 2$.
	\end{lemma}
	\begin{proof}
	We test the first equation in (\ref{problem_a}) with $e^{\chi w} a^{p-1}$ and apply partial integration to see that 
	\begin{align*}
	\frac{1}{p}\frac{\d}{\d t}\int_\Omega e^{\chi w} \a^p &=  \int_\Omega e^{\chi w} \a^{p-1} {\a}_t + \frac{\chi}{p} \int_\Omega {w}_t e^{\chi w} \a^p 
	=  \int_\Omega e^{\chi w} \a^{p-1} {\a}_t - \frac{\chi}{p} \int_\Omega w e^{2 \chi w} \a^{p + 1}\\
	&= \int_\Omega \a^{p-1} \div ( e^{\chi w}\D \grad \a) + \int_\Omega \a^{p-1} \div (a e^{\chi w} (\div \D)) \\ &\hphantom{=\,}+ \logc  \int_\Omega e^{\chi w} \a^{p} - \logc \int_\Omega e^{\loge \chi w} a^{p-1+\loge} + \chi \frac{p-1}{p} \int_\Omega w e^{2 \chi w} \a^{p + 1} \\
	&= - (p-1) \int_\Omega e^{\chi w} \a^{p-2} (\grad \a \cdot \D \grad \a) - (p-1)\int_\Omega  e^{\chi w}\a^{p-1} ((\div \D) \cdot \grad \a) \\&\hphantom{=\,}+ \logc  \int_\Omega e^{\chi w} \a^{p} - \logc \int_\Omega e^{\loge \chi w} a^{p-1+\loge} + \chi \frac{p-1}{p} \int_\Omega w e^{2 \chi w} \a^{p + 1}
	\numberthis \label{eq:ae_test_1}
	\end{align*}
	for all $t \in (0,\tmax)$ and $p \geq 2$. Given our assumptions for $\D$ in (\ref{eq:classical_D_assumption}), we can use Young's inequality to further estimate that
	\begin{align*}
		&- (p-1) \int_\Omega e^{\chi w} \a^{p-2} (\grad \a \cdot \D \grad \a) - (p-1)\int_\Omega  e^{\chi w}\a^{p-1} ((\div \D) \cdot \grad \a) \\
		&\leq -\frac{p-1}{M} \int_\Omega e^{\chi w} a^{p-2} |\grad a|^2 + M(p-1) \int_\Omega e^{\chi w} a^{p-1} |\grad a|
		\\
		&\leq  -\frac{p-1}{2M} \int_\Omega e^{\chi w} a^{p-2} |\grad a|^2 + 2 M^3(p-1) \int_\Omega e^{\chi w} a^p  \\
		&\leq- \frac{p-1}{p^2} \frac{2}{M} \int_\Omega e^{\chi w} |\grad a^\frac{p}{2}|^2 + 2 M^3\, p \int_\Omega e^{\chi w} a^p \\
		&\leq-\frac{1}{p}\frac{1}{M} \int_\Omega e^{\chi w} |\grad a^\frac{p}{2}|^2 + 2 M^3\, p \int_\Omega e^{\chi w} a^p 
	\end{align*}
	as well as more elementary that
	\[
		\chi \frac{p-1}{p} \int_\Omega w e^{2\chi w}a^{p+1} \leq \chi \|w_0\|_\L{\infty} \int_\Omega e^{2\chi w}a^{p+1}
	\]
	for all $t \in (0,\tmax)$ and $p \geq 2$, which when applied to (\ref{eq:ae_test_1}) results in
	\begin{align*}
		&\frac{1}{p}\frac{\d}{\d t}\int_\Omega e^{\chi w} \a^p + \frac{1}{p} \frac{1}{M} \int_\Omega e^{\chi w} |\grad a^\frac{p}{2}|^2 \\		
		\leq& (\logc + 2M^3\, p)  \int_\Omega e^{\chi w} \a^{p} - \logc \int_\Omega e^{\loge \chi w} a^{p-1+\loge} + \chi \|w_0\|_\L{\infty} \int_\Omega e^{2 \chi w} \a^{p + 1} \numberthis \label{eq:ae_test_2}
	\end{align*}
	for all $t \in (0,\tmax)$ and $p \geq 2$. If $r > 2$, we can now further estimate that
	\begin{align*}
	&\hphantom{\leq\;\,} - \logc \int_\Omega e^{\loge \chi w} a^{p-1+\loge} + \chi \|w_0\|_\L{\infty} \int_\Omega e^{2 \chi w} \a^{p + 1} \\
	&\leq - \logc \int_\Omega e^{\loge \chi w} a^{p-1+\loge} + \chi \|w_0\|_\L{\infty} \int_\Omega e^{\loge \chi w} \a^{p + 1}\\
	&\leq \chi \|w_0\|_\L{\infty}  \left(  \frac{\chi \|w_0\|_\L{\infty} }{\mu}\right)^\frac{p + 1}{r-2} e^{\loge \chi \|w_0\|_\L{\infty}} |\Omega| \leq K_1^{p+1}
	\end{align*}
	with \[
		K_1 \defs \left(\chi \|w_0\|_\L{\infty} e^{r\chi \|w_0\|_\L{\infty}} |\Omega| + 1 \right)\left(  \frac{\chi \|w_0\|_\L{\infty} }{\mu}\right)^\frac{1}{r-2}
	\]
	for all $t \in (0,\tmax)$ and $p \geq 2$ by Young's inequality. If, however, $r = 2$ and $\logc \geq \chi \|w_0\|_\L{\infty}$, it is immediately obvious that
	\[
		\hphantom{\leq\;\,} - \logc \int_\Omega e^{\loge \chi w} a^{p-1+\loge} + \chi \|w_0\|_\L{\infty} \int_\Omega e^{2 \chi w} \a^{p + 1} \leq 0 \leq K_1^{p+1}
	\]
	with $K_1 \defs 1$ for all $t \in (0,\tmax)$ and $p \geq 2$. As such, we can in both cases conclude from (\ref{eq:ae_test_2}) that
	\begin{equation}\label{eq:ae_test_3}
		\frac{1}{p}\frac{\d}{\d t}\int_\Omega e^{\chi w} \a^p + \frac{1}{p} \frac{1}{M} \int_\Omega e^{\chi w} |\grad a^\frac{p}{2}|^2 \leq (\logc  + 2 M^3 \,p)  \int_\Omega e^{\chi w} \a^{p} + K_1^{p+1} \leq p\, K_2 \int_\Omega  \a^{p} + K_1^{p+1}
	\end{equation}
	with $K_2 \defs (\mu + 2 M^3)e^{\chi \|w_0\|_\L{\infty}}$ for all $t \in (0,\tmax)$ and $p \geq 2$.
	\\[0.5em]
	We can now use the Gagliardo--Nirenberg inequality to fix a constant $K_3 > 0$ such that
	\begin{align*}
		\int_\Omega a^p &= \| a^\frac{p}{2} \|^2_\L{2} \leq K_3 \|\grad a^\frac{p}{2}\|_\L{2}^{2\alpha} \| a^\frac{p}{2}\|^{2(1-\alpha)}_\L{1} + K_3\| a^\frac{p}{2} \|^2_\L{1} \\
		&\leq \frac{1}{p^2} \frac{1}{M} \frac{1}{K_2}  \int_\Omega |\grad a^\frac{p}{2}|^2 + ( (p^2 M K_2)^\frac{\alpha}{1-\alpha}K_3^\frac{1}{1-\alpha} + K_3) \left( \int_\Omega a^\frac{p}{2} \right)^2 \\
		&\leq \frac{1}{p^2} \frac{1}{M} \frac{1}{K_2}  \int_\Omega e^{\chi w}|\grad a^\frac{p}{2}|^2 + K_4 p^{K_4} \left( \int_\Omega a^\frac{p}{2} \right)^2
	\end{align*}
	for all $t \in (0,\tmax)$ and $p \geq 2$
	with 
	\[
		\alpha \defs \frac{1}{1+\frac{2}{n}} \in (0,1)
	\]
	and $K_4 \defs \max(\frac{2\alpha}{1-\alpha},  (M K_2)^\frac{\alpha}{1-\alpha}K_3^\frac{1}{1-\alpha} + K_3)$. Applying this to (\ref{eq:ae_test_3}) then implies
	\[
		\frac{\d}{\d t}\int_\Omega e^{\chi w} \a^p \leq K_2 K_4 p^{K_4 + 2} \left(\int_\Omega a^\frac{p}{2}\right)^2 + pK_1^{p+1} \leq K_2 K_4 p^{K_4 + 2} \left(\int_\Omega a^\frac{p}{2}\right)^2 + (2K_1)^{p+1} 
	\]
	for all $t \in (0,\tmax)$ and $p \geq 2$.
	Time integration then yields
	\[
		\int_\Omega \a^p(\cdot, t) \leq \int_\Omega e^{\chi w} \a^p(\cdot,t) \leq \tmax K_2 K_4 p^{K_4 + 2} \left( \sup_{s\in(0,\tmax)}\int_\Omega a^\frac{p}{2}(\cdot, s)\right)^2 + \tmax(2K_1)^{p+1} + e^{\chi \|w_0\|_\L{\infty}} \int_\Omega a_0^p
	\]
	for all $t\in(0,\tmax)$ and $p \geq 2$ as $\tmax < \infty$,
	which after estimating the sum on the right-hand side by thrice the maximum of its summands completes the proof.
	\end{proof}\noindent
	We will now proceed to give the actual iteration argument yielding an $L^\infty(\Omega)$-type bound for $a$ and therefore $u$, which is sufficient to rule out finite-time blow-up for the first solution component $u$.
	\begin{lemma}\label{lemma:classical_recursion}
	If $\tmax < \infty$ and further $r > 2$ or $\logc \geq \chi \|w_0\|_\L{\infty}$, then there exists a constant $C > 0$ such that
	\[
	  \|a(\cdot, t)\|_\L{\infty} \leq C \stext{ and therefore } \|u(\cdot, t)\|_\L{\infty} \leq C
	\]
	for all $t\in(0,\tmax)$.
	\end{lemma}
	\begin{proof}
		\newcommand{\iter}{J}
		Let $p_i \defs 2^{i}$, $i \in \N_0$, and $\iter_i \defs \sup_{t\in(0,\tmax)} \left(\int_\Omega a^{p_i}(\cdot,t) \right)^\frac{1}{p_i}$. Then $\iter_0$ is finite because of \Cref{corollary:absolute_baseline_a} and the fact that $p_0 = 1$. We further know that
		\[
		\|a_0\|_\L{p_i} \leq (1+|\Omega|)\|a_0\|_\L{\infty} \sfed K_1.
		\]
		Due to \Cref{lemma:linfty_recursion}, we can conclude that there exists a constant $K_2 \geq 1$ such that the numbers $\iter_i$ conform to the following recursion:
		\[
		\iter_i \leq K_2^\frac{1}{p_i} \max\left(\; \|a_0\|_\L{p_i},\; K_2^{\frac{p_i + 1}{p_i}}, \; p_i^\frac{K_2}{p_i} \iter_{i-1} \right) \;\;\;\; \text{ for all } i \in \N.
		\]
		Iterating this recursion finitely many times ensures that all $\iter_i$ are finite.
		\\[0.5em]
		If there exists an incrementing sequence of indices $i\in\N$, along of which $\iter_i \leq \max(K_1 K_2, K_2^3)$, we immediately gain our desired result by taking the limit of $\iter_i$ along said sequence. As such, we can now assume that there exists $i_0 \in \N$ with
		\[
		\iter_i \geq \max(K_1 K_2, K_2^3)> \left\{
		\begin{aligned}
			&K_2^\frac{1}{p_i}  \|a_0\|_\L{p_i} \\
			&K_2^\frac{1}{p_i} K_2^\frac{p_i+1}{p_i}
		\end{aligned}
		\right.  \;\;\;\; \text{ for all } i \geq i_0 
		\]
		to cover the remaining case. Given these assumptions, the above recursion simplifies to
		\[
		\iter_i \leq (p_i K_2)^\frac{K_2}{p_i}\iter_{i-1}  \leq K_3^{\frac{1}{\sqrt{p_i}}} \iter_{i-1}
		\]
		for all $i \geq i_0$ with some $K_3 > 0$ (only depending on $K_2$) as the function $z \mapsto (zK_2)^\frac{ K_2}{\sqrt{z}}$ is bounded on $[1,\infty)$. By now again iterating this recursion finitely many times, we gain that
		\begin{equation}\label{eq:recusion_consequence}
		\iter_i \leq K_3^{\sum_{j = i_0}^{i}\frac{1}{\sqrt{p_j}}} \iter_{i_0 - 1}
		\end{equation}
		for all $i \geq i_0$.
		As
		\[
		\sum_{j=i_0}^i  \frac{1}{\sqrt{p_j}} = \sum_{j=i_0}^i \left(\frac{1}{\sqrt{2}} \right)^{j} \leq  \sum_{j=0}^\infty \left(\frac{1}{\sqrt{2}} \right)^j < \infty
		\]
		for all $i \geq i_0$ due to the series on the right side being of geometric type, we can conclude from (\ref{eq:recusion_consequence}) that the sequence $J_i$ is uniformly bounded. Therefore, taking the limit $i \rightarrow \infty$ gives us our desired bound for $a$. As $u = a e^{\chi w}$, the corresponding bound for $u$ follows directly from this and \Cref{lemma:absolute_baseline}.
	\end{proof}
	\noindent
	To now establish that finite-time blow-up of the second solution component $w$ is equally as impossible, we will begin by testing the first equation in (\ref{problem_a}) with $-\div (\D \grad a)$ and combining the result with the differential equation associated with $\frac{\d}{\d t}\int_\Omega |\grad w|^4$. The key to extracting a sufficiently strong bound for $w$ is to then use the strength of the absorptive terms originating from the fully elliptic operator $-\div (\D \grad \cdot )$ to counteract the influence of potentially destabilizing terms due to the haptotaxis interaction. Note that the ellipticity of the operator is ensured because we assume that $\D$ is positive definite everywhere in $\overline{\Omega}$.
	\begin{lemma}\label{lemma:grad_w_bound}
		If $\tmax < \infty$ and further $r > 2$ or $\logc \geq \chi \|w_0\|_\L{\infty}$, then there exists a constant $C > 0$ such that
		\[
		\| \grad w(\cdot, t) \|_\L{4} \leq C
		\]	
		for all $t \in (0,\tmax)$.
	\end{lemma}
	\begin{proof}
		Given \Cref{lemma:classical_recursion}, we can fix a constant $K_1 \geq 1$ such that
		\begin{equation}
		\|a(\cdot, t)\|_\L{\infty} \leq K_1 \stext{ and } \int_\Omega \left( a^2(\cdot, t) + a^{2\loge}(\cdot, t) + a^4(\cdot, t) \right) \leq K_1 \;\;\;\; \text{ for all } t \in (0,\tmax). \label{eq:combined_ae_lp_estimates}
		\end{equation}
		Using the Gagliardo--Nirenberg inequality and standard regularity estimates (cf.\ \cite[Theorem 19.1]{FriedmanPartialDifferentialEquations1969} or \cite[Theorem 3.1.1]{LunardiAnalyticSemigroupsOptimal1995})  for the elliptic operator operator $-\div(\D \grad \,\cdot \,)$ (with Neumann-type boundary conditions), we can fix a constant $K_2 \geq 1$ such that
		\[
		\int_\Omega |\grad \phi|^4 \leq K_2  \left( \int_\Omega |\div (\D \grad \phi)|^2  + \int_\Omega |\phi|^2 \right) \|\phi\|^2_\L{\infty} \;\;\;\; \text{ for all } \phi \in C^2(\overline{\Omega}) \text{ with }(\D\grad \phi) \cdot \nu = 0 \text{ on } \partial \Omega.
		\]
		This in turn implies that 
		\begin{equation}
		\int_\Omega |\grad a|^4 \leq K_3 \left( \int_\Omega |\div (\D \grad a)|^2 + 1\right) \label{eq:ae_adapted_gni}
		\end{equation}
		for all $t \in (0,\tmax)$ with $K_3 \defs K_1^3 K_2$ .
		\\[0.5em]
		After establishing these preliminaries, we now note that the first equation in (\ref{problem_a}) can also be written as
		\[
			a_t = \div (\D \grad a) + \chi\grad w \cdot \D \grad a + \div (a (\div \D)) + \chi a (\grad w \cdot (\div \D)) + \mu a (1-a^{r-1}e^{\chi(r-1)w}) +\chi a^2 w e^{\chi w}.
		\]
		We then test this variant of said equation with $-\div (\D \grad \a)$ and employ partial integration (using the fact that $(\div \D) \cdot \nu = 0$ on $\partial \Omega$) as well as Young's inequality to conclude that
		\begin{align*}
		\frac{1}{2}\frac{\d}{\d t} \int_\Omega (\grad a \cdot \D \grad a) &= \int_\Omega (\grad {a}_t \cdot \D \grad a) \\
		&= \int_\Omega \grad ( \div (\D \grad a)) \cdot \D \grad a + \chi \int_\Omega \grad ( \grad w \cdot \D \grad a) \cdot \D \grad a\\
		&\hphantom{=\;}+ \int_\Omega \grad (\div(a \div \D)) \cdot \D \grad a + \chi \int_\Omega \grad (a \grad w \cdot (\div \D)) \cdot \D \grad a \\
		&\hphantom{=\;}+ \int_\Omega \grad \left( \logc  a(1-a^{\loge - 1}e^{\chi(r-1)w} ) + \chi a^2 w e^{\chi w}  \right) \cdot \D \grad a \\
		&\leq -\frac{1}{2}\int_\Omega |\div (\D \grad a) |^2 + 2\chi^2 \int_\Omega |\grad w \cdot \D \grad w| |\grad a \cdot \D \grad a| \\
		&\hphantom{=\;}+ 2\int_\Omega |\div (a \div \D)|^2 + 2\chi^2 \int_\Omega a^2 |\grad w|^2 |\div \D|^2 \\
		&\hphantom{=\;}+ K_4 \int_\Omega \left( a^2 + a^{2\loge} + a^4 \right) \numberthis \label{eq:div_ae_test}
		\end{align*}
		for all $t \in (0,\tmax)$ with $K_4 \defs 8 \max\left( \logc, \logc e^{\chi (\loge - 1)\|w_{ 0}\|_\L{\infty}}, \chi \|w_{0}\|_\L{\infty}e^{\chi \|w_{0}\|_\L{\infty}}\right)^2$. Using the bounds outlined in (\ref{eq:classical_D_assumption}) and (\ref{eq:combined_ae_lp_estimates}), we can now further derive that
		\begin{align*}
		2\chi^2\int_\Omega |\grad w \cdot \D \grad w| |\grad a \cdot \D \grad a| \leq 2\chi^2 M^2\int_\Omega |\grad w|^2 |\grad a|^2 \leq 8 \chi^4 M^4 K_3 \int_\Omega |\grad w|^4 + \frac{1}{8K_3}\int_\Omega |\grad a|^4 
		\end{align*}
		and
		\begin{align*}
		2\int_\Omega |\div (a \div \D)|^2 &\leq 4\int_\Omega |\grad a|^2 |\div \D|^2 + 4 \int_\Omega a^2|\div (\div \D)|^2 \\ &\leq 4 M^2 \left(\int_\Omega |\grad a|^2  + \int_\Omega a^2 \right) \\
		&\leq 4 M^2 \left(\int_\Omega |\grad a|^2  + K_1 \right) \\ 
		&\leq \frac{1}{8K_3}\int_\Omega |\grad a|^4 + 32 M^4K_3 + 4 M^2 K_1
		\end{align*}
		and
		\begin{align*}
		2 \chi^2 \int_\Omega a^2 |\grad w|^2 |\div \D|^2 \leq \chi^2 M^2 \left(\int_\Omega a^4 + \int_\Omega |\grad w|^4 \right) \leq  \chi^2 M^2 K_1 \left(\int_\Omega |\grad w|^4 + 1 \right)
		\end{align*}
		for all $t \in (0,\tmax)$. Applying these three estimates combined with the second bound in (\ref{eq:combined_ae_lp_estimates}) to (\ref{eq:div_ae_test}) then yields
		\begin{equation} \label{eq:div_ae_test_2}
		\frac{1}{2}\frac{\d}{\d t} \int_\Omega (\grad a \cdot \D \grad a) \leq -\frac{1}{2}\int_\Omega |\div (\D \grad a) |^2 + \frac{1}{4K_3} \int_\Omega |\grad a|^4 + K_5 \int_\Omega |\grad w|^4 + K_6 
		\end{equation}
		for all $t \in (0,\tmax)$ with $K_5 \defs 8 \chi^4 M^4 K_3 + \chi^2 M^2 K_1$ and $K_6 \defs 32 M^4K_3 + 4 M^2 K_1 + \chi^2 M^2 K_1 + K_1K_4$.
		\\[0.5em]
		As our second step, we now obtain the following estimate for the time derivative of certain gradient terms of the second solution component $w$ as follows:
		\begin{align*}
		\frac{1}{4} \frac{\d}{\d t}\int_\Omega |\grad w|^4 &= \int_\Omega |\grad w|^2 \grad w \cdot \grad {w}_t = -\int_\Omega |\grad w|^2 \grad w \cdot \grad (a e^{\chi w}w)  \\
		&=-\int_\Omega |\grad w|^4 a e^{\chi w}(\chi w + 1) - \int_\Omega |\grad w|^2 (\grad w \cdot \grad a)  e^{\chi w}w \\
		&\leq K_7 \int_\Omega |\grad w|^3 |\grad a| \leq K_7 \int_\Omega |\grad w|^4 +   K_7 \int_\Omega |\grad a|^4
		\end{align*}
		for all $t \in (0,\tmax)$ with $K_7 \defs \|w_{0}\|_\L{\infty}e^{\chi \|w_{0}\|_\L{\infty}}$.
		\\[0.5em]
		Now combining this with (\ref{eq:div_ae_test_2}) (using an appropriate scaling factor) we gain
		\begin{align*}
		&\frac{1}{2}\frac{\d}{\d t} \int_\Omega (\grad a \cdot \D \grad a) + \frac{1}{16 K_3 K_7 } \frac{\d}{\d t}\int_\Omega |\grad w|^4 \leq -\frac{1}{2}\int_\Omega |\div (\D \grad a) |^2 + \frac{1}{2K_3} \int_\Omega |\grad a|^4 + K_8 \int_\Omega |\grad w|^4 + K_6
		\end{align*}
		for all $t \in (0,\tmax)$ with $K_8 \defs K_5 + \frac{1}{4K_3}$. The application of (\ref{eq:ae_adapted_gni}) to the inequality above then yields
		\begin{align*}
		\frac{1}{2}\frac{\d}{\d t} \int_\Omega (\grad a \cdot \D \grad a) + \frac{1}{16 K_3 K_7} \frac{\d}{\d t}\int_\Omega |\grad w|^4 &\leq K_8 \int_\Omega |\grad w|^4 + K_6 + \frac{1}{2} \\
		&\leq K_9 \left( \frac{1}{2}\int_\Omega (\grad a \cdot \D \grad a) + \frac{1}{16 K_3 K_7}  \int_\Omega |\grad w|^4 \right) + K_6 + \frac{1}{2}
		\end{align*}
		with $K_9 \defs 16 K_3 K_7 K_8$ for all $t\in(0,\tmax)$, which, by a standard comparison argument and the assumption that $\tmax$ is finite, directly gives us our desired result.
	\end{proof}
	\begin{remark}
	The result of the above lemma only ensures that finite-time blow-up of the second solution component is impossible in two and three dimensions according to our blow-up criterion (\ref{eq:blowup}). As such, it is at this point and only this point in this section, where our restriction to two or three dimensions becomes necessary. This, of course, in turn means that any extension of the results of this section to a higher dimensional setting would only need to extend the above argument to one providing better bounds for the gradient of $w$.
	\end{remark}
	\noindent Given that \Cref{lemma:classical_recursion} and \Cref{lemma:grad_w_bound} rule out any kind of finite-time blow-up for our local solutions, the proof of the first central result of this paper can now be stated quite succinctly.
	\begin{proof}[Proof of \Cref{theorem:classical_solution}]
		If we assume $\tmax < \infty$, \Cref{lemma:classical_recursion} and \Cref{lemma:grad_w_bound} in combination contradict the consequence of the blow-up criterion (\ref{eq:blowup}) in this case. Therefore, $\tmax = \infty$ and thus the local solutions constructed in \Cref{lemma:local_solution} must be in fact global. This is sufficient to prove \Cref{theorem:classical_solution} as the fixed assumptions of this section were in fact identical to those of said theorem.
	\end{proof}
	\begin{remark}
		It is also possible to construct classical solutions in the two dimensional case without relying on logistic influences by using some methods that have previously been used when for example dealing with standard diffusion and some slightly modified versions of our arguments (cf.\ \cite{BellomoMathematicalTheoryKeller2015}). 
		\\[0.5em]
		Essentially, the argument boils down to using an estimate of the form
		\[
			\|u\|^3_\L{3} \leq \eps \|u\|^2_{W^{1,2}(\Omega)} \|u\ln(u) \|_\L{1} + C(\eps) \|u\|_\L{1}
		\] 
		with $\eps$ being potentially arbitrarily small (cf.\ \cite[p.1199]{BilerDebyeSystemExistence1994}) in combination with an additional baseline $\int_\Omega u\ln(u)$ estimate based on an energy-type inequality (cf.\ \Cref{lemma:energy_baseline}) to establish an $L^2(\Omega)$ estimate. From there, the arguments are very similar to the Moser-type iteration argument presented above, only with some slight complications added, which are easily surmountable. \Cref{lemma:grad_w_bound} translates basically verbatim.
		\\[0.5em]
		We decided not to present this result here as it will not be needed for our later construction of weak solutions and is not appreciably different from what we have done here or has already been done in the classical diffusion case. 
	\end{remark}
	
	\section{Existence of Weak Solutions}
	
	We have at this point established all the classical existence theory we want to address in this paper and therefore will now transition to our construction of weak solutions, which is in part based on said classical theory.
	
	\subsection{Approximate Solutions}
	
	As is fairly common, our construction of weak solutions will centrally rely on approximation of said solutions by classical solutions, which solve a suitably regularized version of the original problem. As we already derived global existence of classical solutions for the system (\ref{problem}) with very strong assumptions on $\D$, we of course want to construct our weak solutions under much weaker assumptions on $\D$ because there would be almost nothing gained otherwise. As such, the central regularization employed by us will be concerned with approximating a potentially quite irregular $\D$ by matrices $\De$ that are sufficiently regular to ensure classical existence of solutions. Apart from this, we will use approximated initial data. We will also slightly modify the logistic source term to ensure $r > 2$ in our approximated system because we can then further eliminate the assumption concerning the parameters $\chi$ and $\mu$ needed for the classical theory when $r = 2$.
	One central advantage of this approach is that our approximate systems are very close to the system we actually want to construct solutions for and thus our regularizations only minimally interfere with the structures present in the system, which we want to exploit for e.g.\ a priori information.
	\\[0.5em]
	To now make all of this more explicit, we begin by fixing a smooth bounded domain $\Omega\subseteq\R^n$, $n\in\{2,3\}$, and system parameters $\chi \in (0,\infty)$, $\logc \in (0,\infty)$, $\loge \in [2,\infty)$. We also fix some a.e.\ non-negative initial data $u_0 \in L^{z[\ln(z)]_+}(\Omega)$ and $w_0 \in C^0(\overline{\Omega})$ with $\sqrt{w_0} \in W^{1,2}(\Omega)$, where $L^{z[\ln(z)]_+}(\Omega)$ is the standard Orlicz spaces associated with the function $z \mapsto z \left[\ln(z)\right]_+$. We further fix $\D \in W^{1,2}_\mathrm{div}(\Omega;\R^{n\times n}) \cap C^0(\overline{\Omega}; \R^{n\times n})$
	with the following properties: 
	\begin{itemize}
		\item $\D$ is positive semidefinite everywhere. 
		\item $\D$ allows for a divergence estimate with exponent $\beta \in [\frac{1}{2}, 1)$ and constant $A > 0$ such that $\frac{\beta}{1-\beta}\leq r$ (cf.\ \Cref{definition:div_regularity}). 
		\item $\D$ allows for a compact $L^1(\Omega)$ embedding (cf.\ \Cref{definition:comp_regularity}). 
	\end{itemize}
	As for any $\beta \in [\frac{1}{2}, \frac{2}{3}]$ the condition $\frac{\beta}{1-\beta} \leq r$ is always fulfilled independent of our choice of $r \in [2,\infty)$ and as it is easy to see that, if $\D$ allows for a divergence estimate in accordance with \Cref{definition:div_regularity}, it also allows for a divergence estimate with any larger exponent, we can assume that the parameter $\beta$ seen in the second of the above properties is in fact an element of $[\frac{2}{3}, 1) \subseteq (\frac{1}{2}, 1)$ without loss of generality. Then according to \Cref{remark:div_regularity_consequence}, the aforementioned divergence estimate directly implies that 
	\[
		\D \in W^{1,q}_\mathrm{div}(\Omega;\R^{n\times n}) \subseteq  W^{1,2}_\mathrm{div}(\Omega;\R^{n\times n})\subseteq W^{1,\frac{r}{r-1}}_\mathrm{div}(\Omega;\R^{n\times n})
	\]
	with $q \defs \frac{2\beta}{2\beta - 1}$.
	\\[0.5em]
	Given these assumptions, we now choose an approximate family $(\De)_{\eps \in (0,1)} \subseteq C^2(\overline{\Omega}; \R^{n \times n})$ with $\De$ positive definite on $\overline{\Omega}$, $(\div \De) \cdot \nu = 0$ on $\partial \Omega$ for all $\eps \in (0,1)$ and
	\begin{equation}\label{eq:De_convergence}
	\De \rightarrow \D  \stext{in} W^{1,q}_\mathrm{div}(\Omega;\R^{n\times n}) \cap C^0(\overline{\Omega}; \R^{n\times n}) \;\;\;\; \text{ as } \eps \searrow 0.
	\end{equation}
	We can further choose this family in such a way as to ensure that
	\begin{equation}\label{eq:De_div_estimate}
	\int_\Omega |(\div \De) \cdot \Phi| \leq B \left(\int_\Omega (\Phi \cdot \De \Phi)^\beta  + 1\right)
	\end{equation}
	with $B \defs A + 1$ and
	\begin{equation}\label{eq:De_estimate}
	\D + \eps \leq \De \leq \D + 3\eps
	\end{equation}
	for all $\Phi \in C^0(\overline{\Omega}; \R^n)$ and $\eps \in (0,1)$. These additional properties for the approximation $\De$ essentially mean that the regularity properties assumed for $\D$ are also valid for said approximation in an $\eps$ independent fashion.
	\begin{remark}
		Let us briefly illustrate how such an approximation of $\D = (d_{i,j})_{i,j \in \{1,\dots, n\}}$ can be achieved. This will be a two-step process. We first approximate $\D$ in our desired function space with the appropriate boundary conditions and then, as a second step, we show that, with only slight modification, we can gain the remaining properties from that approximation.
		\\[0.5em]
		For the initial approximation, we assume without loss of generality that $\D$ is smooth. We can do this as it is well-known that a standard convolution argument would give us a smooth approximation of $\D$ in our desired space, which we can then approximate again to gain all additional desired properties. In our case, the key property not covered by such a convolution based method is that we want all our approximate matrices to have very specific boundary values. As such, we will now demonstrate how an approximation of a smooth $\D$ by matrices with exactly this property can be achieved using the continuity properties of semigroups associated with carefully chosen sectorial operators (cf.\ \cite{FeffermanSimultaneousApproximationLebesgue2021}). 
		\\[0.5em]
		To this end, we fix functions $d'_{i,j}$ such that 
		\begin{equation}\label{eq:matrix_modification}
		d_{i,j} = \begin{cases} 
		d'_{i,i} + \sum^n_{l,k = 1} d'_{l,k}, &\text{ if } i = j, \\
		d'_{i,j}, &\text{ if } i\neq j
		\end{cases}
		\end{equation}
		for all $i,j \in \{1,\dots,n\}$. 
		%
		% Set $d'_{i,j} = d_{i,j}$ if $i \neq j$. And set $d'_{i,i} = d_{i,i} - \frac{1}{n+1}  \sum_{j,k =1}^n d_{k,j}$
		%
		As can be easily seen, the functions $d'_{i,j}$ are linear combinations of the components of $\D$ and therefore smooth as well. We then set $d'_{i,j,\eps} = e^{\eps L_{i,j}} d'_{i,j}$, $\eps \in (0,1)$, where $L_{i,j}$ is the negative Laplacian on $\Omega$ with boundary conditions $\grad \phi \cdot \nu + \frac{1}{2}(\partial_{x_i} \phi) \nu_j + \frac{1}{2}(\partial_{x_j} \phi) \nu_i  = 0$ and $(e^{t L_{i,j}})_{t\geq 0}$ is the associated semigroup. Due to the well-known continuity properties of said semigroup (cf.\ \cite{HenryGeometricTheorySemilinear1981}, \cite{LunardiAnalyticSemigroupsOptimal1995}, \cite{TriebelInterpolationTheoryFunction1978}), we know that
		$d'_{i,j,\eps} \rightarrow d'_{i,j}$ and therefore $d_{i,j,\eps} \rightarrow d_{i,j}$ in $W^{1,q}(\Omega)\cap C^0(\overline{\Omega})$ as $\eps \searrow 0$ with $d_{i,j,\eps}$ defined in an analogous fashion to (\ref{eq:matrix_modification}). Thus, $\De \defs (d_{i,j,\eps})_{i,j \in \{1,\dots, n\}} \rightarrow \D$ in our desired way. Further, 
		\begin{align*}
		(\div \De) \cdot \nu =& \sum^n_{i,j=1} (\partial_{x_j} d_{i,j,\eps}) \nu_i = \sum _{i,j=1, i\neq j}^n  (\partial_{x_j} d_{i,j,\eps}) \nu_i + \sum^n_{i=1} (\partial_{x_i}  d_{i,i,\eps}) \nu_i \\
		=&\sum _{i,j=1, i\neq j}^n  (\partial_{x_j} d'_{i,j,\eps}) \nu_i + \sum^n_{i=1}  \left( \partial_{x_i} d'_{i,i,\eps} +  \sum^n_{l,k = 1} \partial_{x_i}d'_{l,k,\eps} \right) \nu_i \\
		=& \sum_{i,j=1}^n \left(\tfrac{1}{2}(\partial_{x_j} d'_{i,j,\eps}) \nu_i +  \tfrac{1}{2}(\partial_{x_i} d'_{i,j,\eps}) \nu_j \right)  + \sum^n_{l,k =1} \grad d'_{l,k,\eps}  \cdot \nu \\
		=& \sum_{i,j=1}^n \left(  \grad d'_{i,j,\eps} \cdot \nu + \tfrac{1}{2}(\partial_{x_j} d'_{i,j,\eps}) \nu_i +  \tfrac{1}{2}(\partial_{x_i} d'_{i,j,\eps}) \nu_j \right)
		= 0 \numberthis \label{eq:De_boundary_condition_calculation}
		\end{align*}
		on $\partial \Omega$ for all $\eps \in (0,1)$ due to the prescribed boundary conditions of the operators $L_{i,j}$. Thus, we have constructed a suitable approximate family for $\D$ with the correct boundary conditions. 
		\\[0.5em] 
		Having now presented the full argument used to achieve the boundary condition (\ref{eq:De_boundary_condition_calculation}), let us briefly note that we introduced the functions $d'_{i,j}$ to ensure that the operators $L_{i,j}$ have sufficiently non-tangential boundary conditions and are therefore sectorial (cf.\ \cite{LunardiAnalyticSemigroupsOptimal1995}, \cite{TriebelInterpolationTheoryFunction1978}), which is of course necessary for our semigroup based arguments.
		\\[0.5em]
		As our second step, we will now fix one such family of approximations of $\D$ and call it $\De'$, $\eps \in (0,1)$, as we still want to slightly modify it. We can assume that
		\[
		\|\D'_\eps - \D\|_\L{\infty} \leq \eps \stext{ and } \| \div \De' - \div \D\|_\L{\frac{2\beta}{2\beta-1}} \leq \eps^\frac{1}{2}
		\]
		for all $\eps \in (0,1)$ without loss of generality. If we then set $\De \defs \De' + \| \De' - \D \|_\L{\infty} + \eps$, we can ensure that
		\[
		\D + \eps = \De - \De + \D + \eps = \De - \De' + \D  - \|\De' - \D\|_\L{\infty}  \leq \De + \| \De' - \D\|_\L{\infty} - \|\De' - \D\|_\L{\infty} =  \De
		\]
		and
		\[
		\De = \D - \D + \De = \D + \De' - \D + \|\De' - \D\|_\L{\infty} + \eps \leq \D + 2\|\De' - \D\|_\L{\infty} + \eps \leq \D + 3\eps
		\]
		for all $\eps \in (0,1)$ without affecting any of the desired properties that we already derived as we only modify $\D'_\eps$ by adding constants that converge to zero as $\eps \searrow 0$. This gives us (\ref{eq:De_estimate}).
		\\[0.5em]
		To derive the divergence estimate, we first observe that
		\begin{align*}
			\left|\int_\Omega |(\div \De) \cdot \Phi| - \int_\Omega |(\div \D) \cdot \Phi| \right| \leq& \int_\Omega |\div \De - \div \D||\Phi| \leq \|\div \De - \div \D\|_\L{\frac{2\beta}{2\beta - 1}} \|\Phi\|_\L{2\beta} \\
			\leq& \| \eps^\frac{1}{2} \Phi\|_\L{2\beta} \leq \left(\int_\Omega (\Phi \cdot \eps \Phi)^\beta + 1\right) \leq \left(\int_\Omega (\Phi \cdot \De \Phi)^\beta + 1\right)
		\end{align*}
		for all $\eps \in (0,1)$ and $\Phi \in C^0(\overline{\Omega};\R^n)$. We can then further estimate
		\begin{align*}
			\int_\Omega |(\div \De) \cdot \Phi| 
			\leq& \int_\Omega |(\div \D) \cdot \Phi| + \left|\int_\Omega |(\div \De) \cdot \Phi| - \int_\Omega |(\div \D) \cdot \Phi| \right| \\
			\leq& A \left( \int_\Omega (\Phi \cdot \D \Phi)^\beta + 1 \right) + \left( \int_\Omega (\Phi \cdot \De \Phi)^\beta + 1 \right) \\ 
			\leq& (A + 1)\left( \int_\Omega (\Phi \cdot \De \Phi)^\beta + 1 \right) 
		\end{align*}
		for all $\eps \in (0,1)$ and $\Phi \in C^0(\overline{\Omega};\R^n)$ using our assumed divergence estimate for $\D$ and (\ref{eq:De_estimate}). This gives us (\ref{eq:De_div_estimate}) and thus completes the discussion of our construction.
	\end{remark}\noindent
	We will now proceed to construct our approximate initial data. To do this, we first fix families $(u_{0,\eps})_{\eps \in (0,1)}$, $(w'_{0,\eps})_{\eps \in (0,1)} \subseteq C^3(\overline{\Omega})$ of positive functions with $(\De \grad u_{0, \eps})\cdot\nu = (\De \grad w'_{0, \eps})\cdot\nu = 0$ on $\partial \Omega$ and
	\begin{equation*}
	\begin{aligned}
	u_{0,\eps} &\rightarrow u_0 \;\;\;\;&&\text{in } \L{z[\ln(z)]_+},\\ % That this is possible can be found in e.g. Adams Sobolev Spaces
	w'_{0,\eps} &\rightarrow \sqrt{w_0} &&\text{in } W^{1,2}(\Omega) \cap C^0(\overline{\Omega})
	\end{aligned}
	\end{equation*} 
	as $\eps \searrow 0$. These families can again be constructed by using convolutions or by a similar semigroup based method as seen before in the much more challenging case of the family $(\De)_{\eps \in(0,1)}$. Positivity of both families can further be achieved by first approximating the function in a non-negative way, which is a property of both convolution and semigroup based methods, and then adding $\eps$ to the resulting approximation as a secondary step. 
	\\[0.5em]
	We then let $w_{0,\eps} \defs (w'_{0,\eps})^2 \in C^3(\overline{\Omega})$ for all $\eps \in (0,1)$ and, because of the properties already established for the family $(w'_{0,\eps})_{\eps \in (0,1)}$, it is straightforward to derive that $w_{0,\eps} > 0$ on $\overline{\Omega}$, $(\De \grad w_{0, \eps})\cdot\nu = 0$ on $\partial \Omega$ and
	\begin{equation*}
	\begin{aligned}
		w_{0,\eps} &\rightarrow w_0 \;\;\;\; &&\text{in } C^0(\overline{\Omega}),\\
		\sqrt{w_{0,\eps}} &\rightarrow \sqrt{w_0} &&\text{in } W^{1,2}(\Omega) \cap C^0(\overline{\Omega})
	\end{aligned}
\end{equation*} 
	as $\eps \searrow 0$. 
	\\[0.5em]	
	One important consequence of the above approximations is that we can fix a uniform constant $M > 0$ such that 
	\begin{equation}\label{eq:weak_De_bounds}
		\|\div \De\|_\L{2} \leq M, \;\;\;\; \|\De\|_\L{\infty} \leq M
	\end{equation}
	and 
	\begin{equation}\label{eq:weak_initial_data_bounds}
		\int_\Omega u_{0,\eps} \leq M, \;\;\;\; \int_\Omega u_{0,\eps} \ln(u_{0,\eps}) \leq M,\;\;\;\; \|w_{0,\eps}\|_\L{\infty} \leq M, \;\;\;\; \int_\Omega \frac{\grad {w_{0, \eps}} \cdot \De \grad {w_{0. \eps}}}{w_{0,\eps}} \leq M
	\end{equation}
	for all $\eps \in (0,1)$.
	\\[0.5em]
	We then consider the approximate systems
	\begin{equation}\label{approx_problem}
	\left\{
	\begin{aligned}
	u_{\eps t} &= \div (\De \grad \ue + \ue \div \De) - \chi \div (\ue\De \grad \we) + \logc \ue(1-\ue^{\loge + \eps - 1}) \;\;\;\; &&\text{ on } \Omega\times(0,\infty), \\
	w_{\eps t} &= - \ue \we \;\;\;\; &&\text{ on } \Omega\times(0,\infty),  \\
	(\De \grad \ue) \cdot \nu &= \chi (\ue\De \grad \we) \cdot \nu - \ue (\div \De) \cdot \nu\;\;\;\; &&\text{ on } \partial\Omega\times(0,\infty)\\
	\ue(\cdot, 0) &= u_{\eps, 0} > 0, \;\; \we(\cdot, 0) = w_{\eps, 0} > 0 \;\;\;\; &&\text{ on } \Omega
	\end{aligned}
	\right. 
	\end{equation}
	and use our already established classical existence theory from \Cref{theorem:classical_solution} to now fix  positive, global classical solutions $(\ue, \we)$ to the above system for each $\eps \in (0,1)$. Note that as $r + \eps > 2$, we do not need to make additional assumptions on the parameters $\chi$ and $\logc$ to ensure that said existence theory is applicable.  
	\subsection{Uniform A Priori Estimates}
	We will now derive the bounds necessary to ensure compactness of our families of approximate classical solutions in function spaces conducive to the construction of our desired weak solutions to (\ref{problem}) as limits of said approximate solutions along a suitable sequence of $\eps \in (0,1)$. 
	\\[0.5em]	
	Apart from the baseline established in \Cref{lemma:absolute_baseline} for the classical existence theory, which can be easily translated to our approximate solutions in an $\eps$-independent fashion, we will now derive some extended bounds based on an energy-type inequality as an additional baseline for later arguments in this section. This type of energy inequality was already used in the one-dimensional case in \cite{WinklerSingularStructureFormation2018}. 
	\begin{lemma} \label{lemma:energy_baseline}
		For each $T > 0$, there exists a constant $C \equiv C(T) > 0$ such that
		\begin{align*}
		&\int_\Omega \ue \ln(\ue) + \int_\Omega \frac{\grad \we \cdot \De \grad \we }{\we} + \int_0^t \int_\Omega \frac{\grad \ue \cdot \De \grad \ue}{\ue} + \int_0^t\int_\Omega \ue^{r+\eps} \ln(\ue) \leq C
		\end{align*}
		holds for all $t \in (0,T)$ and all $\eps \in (0,1)$.
	\end{lemma}
	\begin{proof}
		Fix $T > 0$.
		\\[0.5em]
		By then testing the first equation in (\ref{approx_problem}) with $\ln(\ue)$ we gain that
		\begin{align*}
		&\hphantom{=\;}\frac{\d}{\d t} \int_\Omega \ue \ln(\ue) - \frac{\d}{\d t}\int_\Omega \ue = \int_\Omega {\ue}_t \ln(\ue)  \\
		&= \int_\Omega  \ln(\ue) \div \left( \De \grad \ue + \ue \div \De \right) - \chi \int_\Omega \ln(\ue) \div (\ue\De \grad \we) + \logc\int_\Omega \ue(1 - \ue^{\loge + \eps - 1}) \ln(\ue) \\
		&= - \int_\Omega \frac{\grad \ue \cdot \De \grad \ue}{\ue} - \int_\Omega (\div \De) \cdot  \grad \ue + \chi \int_\Omega \grad \ue \cdot \De \grad \we + \logc\int_\Omega \ue(1 - \ue^{\loge + \eps - 1}) \ln(\ue) \\
		&\leq - \int_\Omega \frac{\grad \ue \cdot \De \grad \ue}{\ue} + B\int_\Omega (\grad \ue \cdot \De\grad \ue)^\beta + B + \chi \int_\Omega \grad \ue \cdot \De \grad \we + \logc\int_\Omega \ue(1 - \ue^{\loge + \eps - 1}) \ln(\ue) \\
		&\leq  - \frac{1}{2}\int_\Omega \frac{\grad \ue \cdot \De \grad \ue}{\ue} + 2^\frac{\beta}{1-\beta}B^\frac{1}{1-\beta} \int_\Omega \ue^\frac{\beta}{1-\beta} + B + \chi \int_\Omega \grad \ue \cdot \De \grad \we + \logc\int_\Omega \ue(1 - \ue^{\loge + \eps - 1}) \ln(\ue)
		\label{eq:ue_lnue_test} \numberthis
		\end{align*}
		for all $t\in(0,T)$ and $\eps \in (0,1)$ by partial integration, use of the no-flux boundary conditions and the divergence estimate (\ref{eq:De_div_estimate}) combined with Young's inequality. We can then further gain from the second equation in (\ref{approx_problem}) that
		\begin{align*}
		\frac{1}{2}\frac{\d}{\d t}\int_\Omega \frac{\grad \we \cdot \De \grad \we }{\we} &= \int_\Omega \frac{\grad {\we}_t \cdot \De \grad \we}{\we} - \frac{1}{2}\int_\Omega \frac{ {\we}_t (\grad \we \cdot \De \grad \we)}{\we^2} \\
		&= - \int_\Omega \frac{\ue(\grad \we \cdot \De \grad \we)}{\we} -  \int_\Omega \grad \ue \cdot \De \grad \we + \frac{1}{2}\int_\Omega \frac{\ue (\grad \we \cdot \De \grad \we)}{\we} \\
		&=  -\frac{1}{2}\int_\Omega \frac{\ue(\grad \we \cdot \De \grad \we)}{\we} - \int_\Omega \grad \ue \cdot \De \grad \we \leq - \int_\Omega \grad \ue \cdot \De \grad \we  \numberthis \label{eq:grad_we_test}
		\end{align*}
		for all $t\in(0,T)$ and $\eps \in (0,1)$. Combining (\ref{eq:ue_lnue_test}) and (\ref{eq:grad_we_test}) now allows us to further estimate as follows due to the critical $\int_\Omega \grad \ue \cdot \De \grad \we$ terms in both equations neutralizing each other given the correct coefficients:
		\begin{align*}
		&\frac{\d}{\d t}\left\{  \int_\Omega \ue \ln(\ue) - \int_\Omega \ue  + \frac{\chi}{2} \int_\Omega \frac{\grad \we \cdot \De \grad \we }{\we} \right\} + \frac{1}{2}\int_\Omega \frac{\grad \ue \cdot \De \grad \ue}{\ue} \\ 
		&\leq 2^\frac{\divrege}{1-\divrege}B^\frac{1}{1-\beta}\int_\Omega \ue^\frac{\divrege}{1-\divrege}  + B + \logc \int_\Omega \ue \ln(\ue) - \logc \int_\Omega \ue^{\loge+\eps} \ln( \ue ) \numberthis \label{eq:energy_estimate_derivation}
		\end{align*}
		for all $t \in (0,T)$ and $\eps \in (0,1)$. 
		\\[0.5em]
		As $\frac{\divrege}{1-\divrege} \leq \loge$ by assumption, there exists a constant $K > 0$ (independent of $\eps$) such that
		\[
		2^\frac{\divrege}{1-\divrege}B^\frac{1}{1-\divrege} z^{\frac{\divrege}{1-\divrege}} - \frac{\logc}{2} z^{\loge+\eps} \ln(z) \leq 2^\frac{\divrege}{1-\divrege}B^\frac{1}{1-\divrege} z^{\frac{\divrege}{1-\divrege}} - \frac{\logc}{2} z^{\loge} \ln(z) \leq K
		\]
		for all $z \geq 0$ and $\eps \in (0,1)$. Given this, we can then further estimate in (\ref{eq:energy_estimate_derivation}) to see that
		\begin{align*}
		&\frac{\d}{\d t}\left\{  \int_\Omega \ue \ln(\ue) + \frac{\chi}{2} \int_\Omega \frac{\grad \we \cdot \De \grad \we }{\we} \right\} +  \frac{1}{2} \int_\Omega \frac{\grad \ue \cdot \De \grad \ue}{\ue} + \frac{\mu}{2}\int_\Omega \ue^{r+\eps} \ln(\ue) \\
		&\leq \logc \int_\Omega \ue \ln(\ue) + \frac{\d}{\d t}\int_\Omega \ue + K|\Omega| + B
		\end{align*}
		for all $t \in (0,T)$ and $\eps \in (0,1)$. Time integration in combination with Gronwalls inequality and the uniform $L^1(\Omega)$ bound for $\ue$ due to \Cref{lemma:absolute_baseline} as well as the uniform initial data bounds from (\ref{eq:weak_initial_data_bounds}) then yields our desired result as the above differential inequality essentially means that the growth of the considered terms can be at most exponential.
	\end{proof}\noindent
	We now further extract some relevant but straightforward additional bounds for our approximate solutions from the previous lemma.
	\begin{corollary}\label{lemma:basic_bounds_weak}
		For each $T > 0$, there exists $C \equiv C(T) > 0$ such that
		\begin{equation}
		\int_0^T \int_\Omega u_\eps^{r+\eps} \ln(\ue^{r+\eps}) \leq C, \;\;\;\; \int_0^T \int_\Omega u_\eps^r \ln(\ue) \leq C,  \label{eq:basic_bounds_weak_u} 
		\end{equation}
		\begin{equation}
		\int_0^T \|\ue^\frac{1}{2}(\cdot, s)\|_{\WD{2}}^2 \d s \leq C, \;\;\;\; \int_0^T  \|\ue(\cdot, s)\|^\frac{2r}{r+1}_\WD{\frac{2r}{r+1}} \d s \leq \int_0^T \|\ue(\cdot, s)\|^\frac{2r}{r+1}_{W^{1,\frac{2r}{r+1}}_{\De}(\Omega)} \d s \leq C  \label{eq:basic_bounds_weak_grad_u}
		\end{equation}
		and
		\begin{equation}
		\int_0^T \|\we(\cdot, s)\|_\WD{2}^2 \d s \leq \int_0^T \|\we(\cdot, s)\|_{W^{1,2}_{\De}(\Omega)}^2 \d s \leq C \label{eq:basic_bounds_weak_w}
		\end{equation}
		for all $\eps \in (0,1)$.
	\end{corollary}
	\begin{proof}
		Fix $T > 0$.
		\\[0.5em]
		Then given that
		\[
		\int_\Omega \grad \we \cdot \De \grad \we \leq \|\we\|_\L{\infty} \int_\Omega \frac{\grad \we \cdot \De \grad \we}{\we}
		\]
		for all $t \in(0,T)$ and $\eps \in (0,1)$ as well as knowing that $\D \leq \De$ according to (\ref{eq:De_estimate}) for all $\eps \in (0,1)$, \Cref{lemma:energy_baseline} combined with \Cref{lemma:absolute_baseline}
		and (\ref{eq:weak_initial_data_bounds}) yields  (\ref{eq:basic_bounds_weak_w}). For the definition of the relevant spaces see \Cref{definition:spaces}.
		\\[0.5em]
		As 
		\[
		z^{r} \ln(z) \leq z^{r+\eps}\ln(z) \stext{ and } z^{r+\eps} \ln(z^{r+\eps}) = (r+\eps) z^{r+\eps}\ln(z) \leq (r+1)z^{r+\eps}\ln(z) + 1
		\]
		for all $z\geq 0$ and $\eps \in (0,1)$, the result (\ref{eq:basic_bounds_weak_u}) follows directly from \Cref{lemma:energy_baseline}.
		\\[0.5em]
		To address the last remaining result (\ref{eq:basic_bounds_weak_grad_u}), we now note that
		\[
		\int_\Omega \frac{\grad \ue \cdot \De \grad \ue}{\ue} = 4\int_\Omega \grad \ue^\frac{1}{2} \cdot \De \grad \ue^\frac{1}{2}
		\]
		and
		\begin{align*}
		\int_\Omega \left( \grad \ue \cdot \De \grad \ue \right)^\frac{\frac{2r}{r+1}}{2} &=
		\int_\Omega \ue^\frac{r}{r+1} \left( \frac{\grad \ue \cdot \De \grad \ue}{\ue} \right)^\frac{r}{r+1} \\
		&\leq \int_\Omega \ue^r + \int_\Omega \frac{\grad \ue \cdot \De \grad \ue}{\ue} \leq r\int_\Omega \ue^r \ln(\ue) + |\Omega| + \int_\Omega \frac{\grad \ue \cdot \De \grad \ue}{\ue} 
		\end{align*}
		for all $t \in(0,T)$ and $\eps \in (0,1)$ due to Young's inequality. 
		Given this, the result (\ref{eq:basic_bounds_weak_grad_u}) also follows directly from \Cref{lemma:energy_baseline} and the fact that $\D \leq \De$ for all $\eps \in (0,1)$ according to (\ref{eq:De_estimate}).
	\end{proof}\noindent
	By another testing procedure for the first equation in (\ref{approx_problem}), which is very similar to the one already used by us in the proof of \Cref{lemma:energy_baseline}, we will now derive our final preliminary set of bounds for this section.
	\begin{lemma}\label{lemma:additional_baseline}
		For each $T > 0$, there exists a constant $C\equiv C(T) > 0$ such that
		\[
			\int_0^T \int_\Omega \ue^{-\frac{1}{2}} |(\div \De) \cdot \grad \ue| \leq C \stext{ and } \int_0^T \int_\Omega u^{-\frac{3}{2}}\left(\grad \ue \cdot \De \grad \ue\right) \leq C
		\]
		for all $\eps \in (0,1)$.
	\end{lemma}
	\begin{proof}
		Fix $T > 0$.
		\\[0.5em]
		We first note that 
		\begin{align*}
			\int_\Omega \ue^{-\frac{1}{2}} |(\div \De) \cdot \grad \ue| &= 2\int_\Omega |(\div \De) \cdot \grad \ue^\frac{1}{2}| \leq 2B \int_\Omega \left(\grad \ue^\frac{1}{2} \cdot \De \grad \ue^\frac{1}{2}\right)^\beta + 2B  \\
			&= 2B \int_\Omega \left(\frac{\grad \ue \cdot \De \grad \ue}{\ue}\right)^\beta + 2B\leq 2B\int_\Omega \frac{\grad \ue \cdot \De \grad \ue}{\ue} + 2B(1+|\Omega|)\numberthis \label{eq:u_-12+grad_estimate}
		\end{align*}
		for all $t \in (0,T)$ and $\eps \in (0,1)$ due to (\ref{eq:De_div_estimate}). Given the bounds established in \Cref{lemma:energy_baseline}, this directly gives us the first half of our desired result.
		\\[0.5em]
		We then further test the first equation in (\ref{approx_problem}) with $-\ue^{-\frac{1}{2}}$ to derive that
		\begin{align*}
			-2 \frac{\d}{\d t}\int_\Omega \ue^\frac{1}{2} &= -\int_\Omega \ue^{-\frac{1}{2}} u_{\eps t}
			\\
			&= -\int_\Omega  \ue^{-\frac{1}{2}} \div \left( \De \grad \ue + \ue \div \De \right) + \chi \int_\Omega \ue^{-\frac{1}{2}} \div (\ue\De \grad \we) - \logc\int_\Omega \ue^\frac{1}{2}(1 - \ue^{\loge + \eps - 1}) \\
			&= -\frac{1}{2}\int_\Omega \ue^{-\frac{3}{2}} (\grad \ue \cdot \De \grad \ue)  -\frac{1}{2} \int_\Omega \ue^{-\frac{1}{2}}((\div \De) \cdot \grad \ue) 
			\\
			&\hphantom{=\;}+ \frac{\chi}{2} \int_\Omega \ue^{-\frac{1}{2}}(\grad\ue \cdot \De \grad \we) - \logc\int_\Omega \ue^\frac{1}{2}(1 - \ue^{\loge + \eps - 1}) \\
			& \leq -\frac{1}{2}\int_\Omega \ue^{-\frac{3}{2}} (\grad \ue \cdot \De \grad \ue) + \frac{1}{2}\int_\Omega \ue^{-\frac{1}{2}} |(\div \De) \cdot \grad \ue| \\
			&\hphantom{=\;}+ \frac{\chi}{4} \int_\Omega \frac{\grad \ue \cdot \De \grad \ue}{\ue}  + \frac{\chi}{4} \int_\Omega \grad \we \cdot \De \grad \we + \logc\int_\Omega \ue^{r+\eps-\frac{1}{2}}
		\end{align*}
		for all $t\in(0,T)$ and $\eps \in (0,1)$  by partial integration, use of the no-flux boundary conditions and the Cauchy--Schwarz inequality combined with Young's inequality. This then immediately implies
		\begin{align*}
			\frac{1}{2}\int_\Omega \ue^{-\frac{3}{2}} (\grad \ue \cdot \De \grad \ue) &\leq 	2 \frac{\d}{\d t}\int_\Omega \ue^\frac{1}{2}  + \frac{1}{2}\int_\Omega \ue^{-\frac{1}{2}} |(\div \De) \cdot \grad \ue| \\			
			&\hphantom{=\;} + \frac{\chi}{4} \int_\Omega \frac{\grad \ue \cdot \De \grad \ue}{\ue} + \frac{\chi}{4} \int_\Omega \grad \we \cdot \De \grad \we +  \logc \int_\Omega \ue^{r+\eps-\frac{1}{2}}  \numberthis \label{eq:secondary_ue_gradient_estimate}
		\end{align*}
		for all $t\in(0,T)$ and $\eps \in (0,1)$. As further
		\[
			\int_\Omega \ue^{\loge + \eps - \frac{1}{2}} \leq \int_\Omega \ue^{r+\eps} + |\Omega| \leq \int_\Omega \ue^{r+\eps}\ln(\ue^{r+\eps}) + 2|\Omega|
		\]
		for all $t\in(0,T)$ and $\eps \in (0,1)$, the inequality (\ref{eq:secondary_ue_gradient_estimate}) combined with the already established bounds from \Cref{lemma:absolute_baseline}, \Cref{lemma:energy_baseline} and (\ref{eq:weak_initial_data_bounds}) as well as (\ref{eq:u_-12+grad_estimate}) gives us our desired estimate after an integration in time.
	\end{proof}
	\subsection{Construction of Weak Solutions}
	As our final preparation for a now soon following compactness argument  (based on the Aubin--Lions lemma), which is used to construct the candidates for our weak solutions, we will now prepare uniform integrability estimates for the time derivatives of $\ue^{1/2}$ and $\we$. 
	\\[0.5em]
	Note that the construction of a solution candidate for the second solution component $w$ could likely be achieved by less powerful means. But as we will already need to employ fairly extensive compact embedding arguments to handle the first solution components $\ue$ anyway and deriving the necessary additional uniform bounds for $\we$ is trivial, we will use the same compactness argument for the second solution component as well for the sake of uniformity of presentation.
	\begin{lemma}\label{lemma:dual_bounds_weak}
		Then for each $T > 0$, there exists $C \equiv C(T) > 0$ such that
		\[
		\int_0^T \|(\ue^\frac{1}{2})_t(\cdot, t)\|_{(W^{n+1,2}(\Omega))^*} \d t \leq C \stext{ and } \int_0^T \|w_{\eps t}(\cdot, t)\|_{(W^{n+1,2}(\Omega))^*} \d t \leq C 
		\]
		for all $\eps \in (0,1)$.
	\end{lemma}
	\begin{proof}
		Fix $T > 0$.
		\\[0.5em]
		We then begin by noting that the bound for ${\we}_t$ is an immediate and straightforward consequence of \Cref{lemma:absolute_baseline} combined with (\ref{eq:weak_initial_data_bounds}) and the second equation in (\ref{approx_problem}) as well as the fact that $W^{n+1, 2}(\Omega)$ embeds continuously into $L^\infty(\Omega)$. 
		\\[0.5em]
		As such, we now focus our attention on deriving the $(\ue^\frac{1}{2})_t$ bound. 
		To this end, we test the first equation in (\ref{approx_problem}) with $u_\eps^{-\frac{1}{2}} \phi$,  $\phi \in C^\infty(\overline{\Omega})$, and apply partial integration and (\ref{eq:weak_De_bounds}) as well as the Hölder and Young's inequality to see that
		\begin{align*}
		&2\left|\int_\Omega (u^\frac{1}{2}_{\eps})_t \phi\right| 
		= \left|\int_\Omega \ue^{-\frac{1}{2}} u_{\eps t} \phi \right|\\
		&\leq \left|\int_\Omega \ue^{-\frac{1}{2}}\grad \ue \cdot \De \grad \phi \right| + \frac{1}{2}\left|\int_\Omega \ue^{-\frac{3}{2}}(\grad \ue \cdot \De \grad \ue) \phi \right|  \\
		&\hphantom{=\;}+ \left|\int_\Omega \ue^\frac{1}{2} ((\div \De) \cdot \grad \phi) \right| + \frac{1}{2}\left|\int_\Omega \ue^{-\frac{1}{2}} ((\div \De) \cdot \grad \ue) \phi \right| \\
		&\hphantom{=\;}+ 
		\chi\left|\int_\Omega\ue^\frac{1}{2} \grad \we \cdot \De \grad \phi \right| + \frac{\chi}{2}	\left|\int_\Omega\ue^{-\frac{1}{2}} ( \grad \we \cdot \De \grad \ue) \phi \right| + \logc \left| \int_\Omega \ue^\frac{1}{2} ( 1- \ue^{\loge + \eps - 1})\phi \right|\\
		&\leq \left(\int_\Omega \grad \phi \cdot \De \grad \phi \right)^\frac{1}{2}\left(\int_\Omega \frac{\grad \ue \cdot \De \grad \ue}{\ue}\right)^\frac{1}{2}  + \frac{\|\phi\|_\L{\infty}}{2} \int_\Omega \ue^{-\frac{3}{2}}  (\grad \ue \cdot \De \grad \ue) \\
		&\hphantom{=\;}+ \|\grad \phi\|_\L{\infty} \left( \int_\Omega \ue + \int_\Omega |\div \De|^2 \right) + \frac{\|\phi\|_\L{\infty}}{2} \int_\Omega \ue^{-\frac{1}{2}}|(\div \De) \cdot \grad \ue| \\
		&\hphantom{=\;}+ \chi\left( \int_\Omega \ue\grad \phi \cdot \De \grad \phi \right)^\frac{1}{2}\left(\int_\Omega \grad \we \cdot \De \grad \we\right)^\frac{1}{2} \\
		&\hphantom{=\;}+ \frac{\chi \|\phi\|_\L{\infty}}{2} \left(\int_\Omega \frac{\grad \ue \cdot \De \grad \ue}{\ue}\right)^\frac{1}{2}\left(\int_\Omega \grad \we \cdot \De \grad \we \right)^\frac{1}{2}  \\
		&\hphantom{=\;}+\mu \|\phi\|_\L{\infty}\int_\Omega \ue^\frac{1}{2} + \mu  \|\phi\|_\L{\infty} \int_\Omega \ue^{r+\eps - \frac{1}{2}} \\
		&\leq K\left( \|\phi\|_{L^\infty(\Omega)} + \|\grad \phi\|_{L^\infty(\Omega)}\right) \left( \int_\Omega \frac{\grad \ue \cdot \De \grad \ue}{\ue} + \int_\Omega \ue^{-\frac{3}{2}}  (\grad \ue \cdot \De \grad \ue) + \int_\Omega \ue^{-\frac{1}{2}}|(\div \De) \cdot \grad \ue|\right.\\ 
		& \qquad\qquad\qquad\qquad\qquad\qquad \left.\hphantom{=\;\;\;\;}+ \int_\Omega \grad \we \cdot \De \grad \we + \int_\Omega \ue^{r+\eps}\ln(\ue^{r+\eps})+ 1 \right)
		\end{align*}
		for all $t\in(0,T)$ and $\eps \in (0,1)$ with some appropriate constant $K > 0$ only dependent on $\Omega$, $\mu$, $r$, $\chi$ and $M$. Given the above inequality, the remainder of our desired result follows from \Cref{lemma:energy_baseline}, \Cref{lemma:basic_bounds_weak} and \Cref{lemma:additional_baseline} as well as the continuous embedding of $W^{n+1,2}(\Omega)$ into $W^{1,\infty}(\Omega)$ and density of $C^\infty(\overline{\Omega})$ in $W^{n+1,2}(\Omega)$.
	\end{proof}
	\noindent
	Having prepared all the necessary bounds, we will now construct the solution candidates by using various compact embedding arguments to gain them as the limit of our approximate solutions.
	\begin{lemma}\label{lemma:convergence_properties}
		There exist a null sequence $(\eps_j)_{j\in\N}\subseteq (0,1)$ and a.e.\ non-negative functions 
		\begin{align*}
		u &\in L_\loc^\frac{2r}{r+1}([0,\infty);\WD{\frac{2r}{r+1}}) \cap L_\loc^r(\overline{\Omega}\times[0,\infty)), \\
		w &\in L_\loc^2([0,\infty);\WD{2})\cap L^\infty(\Omega\times(0,\infty)),
		\end{align*} 
		such that
		\begin{align}
		\ue &\rightarrow u &&\text{ in } L_\loc^r(\overline{\Omega}\times[0,\infty)) \text{ and a.e.\ in } \Omega\times[0,\infty), \label{eq:ue_convergence}\\
		\ue^{r+\eps} &\rightarrow u^r &&\text{ in }L^1_\loc(\overline{\Omega} \times [0,\infty)) \text{ and a.e.\ in } \Omega\times[0,\infty),\label{eq:ue_higher_r_convergence}\\
		\ue &\rightharpoonup u && \text{ in } L^\frac{2r}{r+1}_\loc([0,\infty); W^{1,\frac{2r}{r+1}}_\D(\Omega)),\label{eq:ue_weak_convergence}  \\
		\we &\rightarrow w &&\text{ in } L_\loc^p(\overline{\Omega}\times[0,\infty)) \text{ for all } p \in [1,\infty) \text{ and a.e.\ in } \Omega\times[0,\infty),\label{eq:we_convergence}  \\
		\we &\rightharpoonup w && \text{ in } L^2_\loc([0,\infty); W^{1,2}_\D(\Omega))\label{eq:we_weak_convergence}  
		\end{align}
		as $\eps = \eps_j \searrow 0$.
	\end{lemma}
	\begin{proof}
		Given that both the families $(\ue^\frac{1}{2})_{\eps\in(0,1)}$ and $(\we)_{\eps\in(0,1)}$ are bounded in $L_\loc^2([0,\infty); \WD{2})$ according to \Cref{lemma:basic_bounds_weak} and the families $(({\ue^\frac{1}{2}})_t)_{\eps\in(0,1)}$ and $({\we}_t)_{\eps\in(0,1)}$ are bounded in $L_\loc^1([0,\infty); (W^{n+1,2}(\Omega))^*)$ according to \Cref{lemma:dual_bounds_weak}, we can apply the Aubin--Lions lemma (cf.\  \cite{TemamNavierStokesEquationsTheory1977}) to the above families using the triple of embedded spaces $\WD{2} \subseteq L^1(\Omega) \subseteq (W^{n+1,2}(\Omega))^*$. Note that this is only possible as the first embedding is in fact compact by our assumptions (cf.\  \Cref{definition:comp_regularity}). Therefore, there exists a null sequence $(\eps_j)_{j\in\N} \subseteq (0,1)$ and functions $\tilde{u},w: \overline{\Omega} \times [0,T) \rightarrow \R$ such that
		\[
		\ue^\frac{1}{2} \rightarrow \tilde{u} \stext{ and } \we \rightarrow w \;\;\;\; \text{ in }  L_\loc^2([0,\infty);L^1(\Omega)) \text{ and therefore in } L_\loc^1(\overline{\Omega}\times[0,\infty))
		\]
		as $\eps = \eps_j \searrow 0$. This sequence is constructed by applying the Aubin--Lions lemma countably infinitely many times on time intervals of the form $[0,T]$, $T\in\N$, combined with a straightforward extension and diagonal sequence argument. We can further choose the above sequence in such way as to ensure that $\ue^\frac{1}{2} \rightarrow \tilde{u}$ and $\we \rightarrow w$ pointwise almost everywhere as $\eps = \eps_j \searrow 0$ by potentially switching to another subsequence. Due to the family $(\we)_{\eps \in (0,1)}$ furthermore being uniformly bounded in $L^\infty(\Omega\times(0,\infty))$ (cf.\ (\ref{eq:weak_initial_data_bounds}) and \Cref{lemma:absolute_baseline}), the above convergence properties directly imply (\ref{eq:we_convergence}) as well as the fact that $w$ is non-negative almost everywhere and $w \in L^\infty(\Omega\times(0,\infty))$.
		\\[0.5em]
		We now set $u \defs \tilde{u}^2$ and observe that the above almost everywhere pointwise convergence for the already constructed sequences then ensures that
		\[
			\ue \rightarrow u \stext{ and } \ue^{r+\eps} \rightarrow u^r \;\;\;\; \text{ a.e.\ pointwise}
		\] 
		as $\eps = \eps_j \searrow 0$. This immediately gives us non-negativity of $u$ as well. Further as for every $T > 0$ there exists $K \equiv K(T) > 0$ such that
		\[
		\int_0^T \int_\Omega \ue^r|\ln(\ue)| \leq K \stext{ and } \int_0^T \int_\Omega \ue^{r+\eps} |\ln(\ue^{r+\eps})| \leq K
		\]
		for all $\eps \in (0,1)$ according to \Cref{lemma:basic_bounds_weak}, we can use Vitali's theorem and the de La Valleé Poussin criterion for uniform integrability (cf.\ \cite[pp.\ 23-24]{DellacherieProbabilitiesPotential1978}) to gain convergence properties (\ref{eq:ue_convergence}), (\ref{eq:ue_higher_r_convergence}).
		\\[0.5em]
		The remaining weak convergence properties (\ref{eq:ue_weak_convergence}) and (\ref{eq:we_weak_convergence}) then follow immediately by another similar but fairly standard subsequence extraction argument as the respective families of functions are bounded in the relevant spaces according to \Cref{lemma:basic_bounds_weak}.
		\\[0.5em]
		As all not yet explicitly established regularity properties for $u$ and $w$ directly follow from the convergence properties and we have at this point proven all said properties, this completes the proof.
	\end{proof} \noindent
	For the remainder of this section, we will now fix the functions $u$, $w$ as well as the sequence $(\eps_j)_{j\in\N}$ constructed in the preceding lemma. While the convergence properties derived in \Cref{lemma:convergence_properties} are in fact already sufficient to allow us to translate the weak solution property from our approximate solutions to our now established solution candidates, we will as a last effort before the proof of \Cref{theorem:weak_solutions} derive some more specifically tailored convergence properties to handle some of the more complex terms in the weak solution definition.
	\\[0.5em]
	Note that the following lemma is the critical point in this section where the assumption $r \geq 2$ becomes important. The only other points in this section where this assumption was used are the argument ensuring the existence of the approximate solutions, where it could likely be dropped by introducing yet more regularizations to (\ref{approx_problem}), and our assumption that $\beta$ is an element of $[\frac{2}{3}, 1)$ without loss of generality, which was done mostly for convenience.
	\begin{lemma} \label{lemma:additional_convergence_properties}
		The convergence properties 
		\begin{equation}\label{eq:weakish_convergence_1}
		\int_0^\infty\int_\Omega \grad \ue \cdot \De \grad \phi \rightarrow \int_0^\infty\int_\Omega \grad u \cdot \D \grad \phi \;\;\;\; \text{ as } \eps = \eps_j \searrow 0
		\end{equation}
		and
		\begin{equation}\label{eq:weakish_convergence_2}
		\int_0^\infty\int_\Omega \ue \grad \we \cdot \De \grad \phi \rightarrow \int_0^\infty\int_\Omega u \grad w \cdot \D \grad \phi \;\;\;\; \text{ as } \eps = \eps_j \searrow 0
		\end{equation}
		hold for all $\phi \in C_c^\infty(\overline{\Omega}\times[0,\infty))$.
	\end{lemma}
	\begin{proof}
		Fix $\phi \in C_c^\infty(\overline{\Omega}\times[0,\infty))$ and $T > 0$ such that $\supp \phi \subseteq \overline{\Omega}\times[0,T)$. We can then fix a constant $K_1 \geq 1$ such that
		\[
		\int_0^T \int_\Omega \left(\grad \ue \cdot \De \grad \ue\right)^\frac{r}{r+1} \leq K_1 \stext{ and } \int_0^T \int_\Omega \grad \we \cdot \De \grad \we \leq  K_1 
		\]
		for all $\eps \in (0,1)$ according to \Cref{lemma:basic_bounds_weak}. This implies that
		\begin{align*}
		\|\grad \ue\|_{L^1(\Omega\times(0,T))} &\leq (|\Omega| + 1)\left(\int_0^T\int_\Omega \left(\grad \ue \cdot \grad \ue\right)^\frac{r}{r+1} \right)^\frac{r+1}{2r} \\
		&\leq (|\Omega| + 1) \left(\frac{1}{\eps^\frac{r}{r+1}}\int_0^T\int_\Omega  \left(\grad\ue \cdot \De \grad \ue \right)^\frac{r}{r+1} \right)^\frac{r+1}{2r} \leq \frac{K_2}{\eps^\frac{1}{2}} \numberthis\label{eq:grad_ue_eps_estimate}
		\end{align*}
		and
		\[
		\|\grad \we\|_{L^2(\Omega\times(0,T))}  =\left(\int_0^T\int_\Omega \grad \we \cdot \grad \we \right)^\frac{1}{2} \leq (|\Omega| + 1) \left(\frac{1}{\eps}\int_0^T\int_\Omega \grad \we \cdot \De \grad \we \right)^\frac{1}{2} \leq \frac{K_2}{\eps^\frac{1}{2}}\numberthis\label{eq:grad_we_eps_estimate}
		\]
		for all $\eps \in (0,1)$ with $K_2 \defs K_1(|\Omega| + 1)$ due to the Hölder inequality and the estimate (\ref{eq:De_estimate}). We then observe that
		\begin{align*}
		&\left| \int_0^T\int_\Omega \grad \ue \cdot \De \grad \phi - \int_0^T\int_\Omega \grad u \cdot \D \grad \phi \right| \\
		&\leq \left| \int_0^T\int_\Omega \grad \ue \cdot \De \grad \phi - \int_0^T\int_\Omega \grad \ue \cdot \D \grad \phi \right| + \left| \int_0^T\int_\Omega \grad \ue \cdot \D \grad \phi - \int_0^T\int_\Omega \grad u \cdot \D \grad \phi \right| \\
		&\leq \|\grad \ue\|_{L^1(\Omega\times(0,T))}\|\De - \D\|_\L{\infty}\|\grad \phi\|_{L^\infty(\Omega\times(0,T))}  + \left|\int_0^T\int_\Omega \grad \ue \cdot \D \grad \phi - \int_0^T\int_\Omega \grad u \cdot \D \grad \phi \right|\\
		&\leq 3 K_2 \eps^\frac{1}{2} \|\grad \phi\|_{L^\infty(\Omega\times(0,T))} + \left| \int_0^T\int_\Omega \grad \ue \cdot \D \grad \phi - \int_0^T\int_\Omega \grad u \cdot \D \grad \phi \right|
		\end{align*}
		for all $\eps \in (0,1)$ because of (\ref{eq:De_estimate}) and (\ref{eq:grad_ue_eps_estimate}). This inequality immediately implies (\ref{eq:weakish_convergence_1}) due to the weak convergence property (\ref{eq:ue_weak_convergence}) presented in \Cref{lemma:convergence_properties}.
		\\[0.5em]
		We now similarly estimate that
		\begin{align*}
		&\left| \int_0^T\int_\Omega \ue \grad \we \cdot \De \grad \phi - \int_0^T\int_\Omega u \grad w \cdot \D \grad \phi \right| \\
		&\leq \left| \int_0^T\int_\Omega \ue \grad \we \cdot \De \grad \phi - \int_0^T\int_\Omega u \grad \we \cdot \De \grad \phi \right| + \left| \int_0^T\int_\Omega u \grad \we \cdot \De \grad \phi - \int_0^T\int_\Omega u \grad \we \cdot \D \grad \phi \right|\\
		&\hphantom{=\;}+ \left| \int_0^T\int_\Omega u \grad \we \cdot \D \grad \phi - \int_0^T\int_\Omega u \grad w \cdot \D \grad \phi \right| \\
		&\leq \left(\int_0^T \int_\Omega \grad \we \cdot \De \grad \we\right)^\frac{1}{2}\left(\int_0^T \int_\Omega (u-\ue)^2 (\grad \phi \cdot \De \grad  \phi)\right)^\frac{1}{2} \\		
		&\hphantom{=\;}+ \|u\|_{L^2(\Omega\times(0,T))} \|\grad \we\|_{L^2(\Omega\times(0,T))} \|\De - \D\|_\L{\infty} \|\grad \phi\|_{L^\infty(\Omega\times(0,T))} \\
		&\hphantom{=\;}+ \left| \int_0^T\int_\Omega u \grad \we \cdot \D \grad \phi - \int_0^T\int_\Omega u \grad w \cdot \D \grad \phi \right| \\
		&\leq K_1^\frac{1}{2} (\|\D\|_\L{\infty} + 3)^\frac{1}{2} \|\grad \phi\|_\L{\infty} \|u - \ue\|_{L^2(\Omega\times(0,T))} 
		+ 3 K_2 \eps^\frac{1}{2} \|u\|_{L^2(\Omega\times(0,T))} \|\grad \phi\|_{L^\infty(\Omega\times(0,T))} \\
		&\hphantom{=\;}+ \left| \int_0^T\int_\Omega u \grad \we \cdot \D \grad \phi - \int_0^T\int_\Omega u \grad w \cdot \D \grad \phi \right| 
		\end{align*}
		for all $\eps \in (0,1)$ because of (\ref{eq:De_estimate}) and (\ref{eq:grad_we_eps_estimate}). Due to the convergence properties (\ref{eq:ue_convergence}) and (\ref{eq:we_weak_convergence}) as well as the fact that $r \geq 2$ and therefore $u \in L_\loc^2(\overline{\Omega}\times[0,\infty))$, the above estimate implies (\ref{eq:weakish_convergence_2}) and thus completes the proof.
	\end{proof}
	\noindent As all convergence properties necessary to argue that $u$ and $w$ are in fact our desired weak solution have been established, we can now present the at this point fairly short proof of our second main existence result, namely \Cref{theorem:weak_solutions}. 
	\begin{proof}[Proof for \Cref{theorem:weak_solutions}]
		We first note that $u$, $w$ are already sufficiently regular to ensure that (\ref{eq:weak_solution_regularity_1}) and (\ref{eq:weak_solution_regularity_2}) hold due to \Cref{lemma:convergence_properties}.
		\\[0.5em]		
		It is further straightforward to verify that $\ue, \we$ are weak solutions in the sense of \Cref{definition:weak_solution} for all $\eps \in (0,1)$ with only slightly different parameters. As such, we only now need to confirm that all the terms in the weak solution definition converge to their counterparts without $\eps$. For all the terms that are structurally identical in both the approximated case as well as in the weak solution definition we want to achieve for $u$ and $w$, this is covered by \Cref{lemma:convergence_properties} as well as the convergence properties of the initial data laid out at the beginning of this section. The terms that differ because $\D$ was replaced by $\De$ are covered by either \Cref{lemma:additional_convergence_properties} or (\ref{eq:De_convergence}) combined with (\ref{eq:ue_convergence}) from \Cref{lemma:convergence_properties}. Finally, the logistic terms $\int_0^\infty\int_\Omega \ue(1-\ue^{r-1+\eps})\phi$ occurring in the weak solution definition for our approximate solutions converge to their proper counterpart $\int_0^\infty\int_\Omega u (1-u^{r-1})\phi$ due to (\ref{eq:ue_convergence}) and (\ref{eq:ue_higher_r_convergence}) from \Cref{lemma:convergence_properties} as well. We have now discussed that all the terms occurring in the weak solution definition of the approximate solutions converge to the correct terms for our solution candidates. Therefore $(u,w)$ is a weak solution of the type described in \Cref{definition:weak_solution}.  
	\end{proof} 
	
	\section*{Acknowledgment} The author acknowledges the support of the \emph{Deutsche Forschungsgemeinschaft} in the context of the project \emph{Emergence of structures and advantages in cross-diffusion systems}, project number 411007140.

\end{document}